\documentclass{article}

\usepackage[T1]{fontenc}
\usepackage{lmodern}

\usepackage{microtype}

\usepackage{graphicx}
\usepackage{booktabs} 

\usepackage{amsfonts}       
\usepackage{fullpage}
\usepackage{amsmath,amsthm}

\usepackage{algorithm}
\usepackage{algorithmic}
\usepackage{float}
\usepackage[T1]{fontenc}
\usepackage{tikz}
\usetikzlibrary{arrows}
\usetikzlibrary{calc}
\usetikzlibrary{snakes}
\usepackage{cite}
\usepackage{graphicx}
\usepackage{array}
\usepackage{mdwmath}
\usepackage{mdwtab}
\usepackage{eqparbox}
\usepackage{fixltx2e}
\usepackage{url}
\usepackage{pgf}
\usepackage{lipsum}
\usepackage{multirow}
\usepackage{xspace}
\usepackage{xcolor}
\usepackage{algorithmic}
\usepackage{bm}
\usepackage{bbm}
\usepackage{booktabs}
\usepackage{etoolbox}
\usepackage{makecell}
\usepackage{nicefrac}
\usepackage{textcomp}
\usepackage{stmaryrd}
\usepackage{mathrsfs}
\usepackage{paralist}
\usepackage{comment}

\usepackage[font=footnotesize,labelsep=space,labelfont=bf]{caption}

\usepackage[colorlinks=true, citecolor=burntumber, linkcolor = bluepigment ,urlcolor=blue, pagebackref=false]{hyperref}

\makeatletter

\newcommand{\Rmnum}[1]{\expandafter\@slowromancap\romannumeral #1@}
\makeatother

\newtheorem{theorem}{Theorem}
\newtheorem*{theorem*}{Theorem}
\newtheorem{definition}{Definition}
\newtheorem{lemma}{Lemma}
\newtheorem*{lemma*}{Lemma}
\newtheorem*{cor*}{Corollary}
\newtheorem{fact}{Fact}

\newcommand{\namedref}[2]{\hyperref[#2]{#1~\ref*{#2}}}

\definecolor{darkred}{rgb}{0.5, 0, 0} 
\definecolor{darkblue}{rgb}{0,0,0.5} 
\hypersetup{
    colorlinks=true,
    linkcolor=darkred,
    citecolor=darkblue,
    urlcolor=darkblue   
}

\newcommand{\verts}[1]{\Vert #1 \Vert}

\newcommand{\projXn}{\ensuremath{{ \Pi_{\mathcal{X}^n} }}\xspace}

\newcommand{\projX}{\ensuremath{{ \Pi_{\mathcal{X}} }}\xspace}
\newcommand{\projXt}{\ensuremath{{ \Pi_{\widetilde{\mathcal{X}}} }}\xspace}

\newcommand{\btheta}{\ensuremath{{ \boldsymbol{\theta} }}\xspace}
\newcommand{\bxi}{\ensuremath{{ \boldsymbol{\xi} }}\xspace}
\newcommand{\blambda}{\ensuremath{{ \boldsymbol{\lambda} }}\xspace}

\newcommand{\be}{\ensuremath{{\bf e}}\xspace}

\newcommand{\bhx}{\ensuremath{\hat{\bf x}}\xspace}

\newcommand{\bg}{\ensuremath{{\bf g}}\xspace}

\newcommand{\bx}{\ensuremath{{\bf x}}\xspace}
\newcommand{\bq}{\ensuremath{{\bf q}}\xspace}
\newcommand{\bp}{\ensuremath{{\bf p}}\xspace}

\newcommand{\btx}{\ensuremath{\widetilde{\bf x}}\xspace}
\newcommand{\blx}{\ensuremath{\overline{\bf x}}\xspace}
\newcommand{\by}{\ensuremath{{\bf y}}\xspace}

\newcommand{\bu}{\ensuremath{{\bf u}}\xspace}
\newcommand{\bv}{\ensuremath{{\bf v}}\xspace}

\newcommand{\C}{\ensuremath{\mathcal{C}}\xspace}
\newcommand{\bbE}{\ensuremath{\mathbb{E}}\xspace}

\newcommand{\cP}{\ensuremath{\mathcal{P}}\xspace}

\newcommand{\calH}{\ensuremath{\mathcal{H}}\xspace}
\newcommand{\cN}{\ensuremath{\mathcal{N}}\xspace}

\newcommand{\cX}{\ensuremath{\mathcal{X}}\xspace}

\newcommand{\trans}{\ensuremath{\mathsf{T}}\xspace}

\renewcommand{\paragraph}[1]{\smallskip\noindent{\bf #1}~}

\restylefloat{table}
\usepackage{caption}
\usepackage[margin=1in]{geometry}
\definecolor{brightmaroon}{rgb}{0.76, 0.13, 0.28}
\definecolor{ceruleanblue}{rgb}{0.16, 0.32, 0.75}
\definecolor{bluepigment}{rgb}{0.2, 0.2, 0.6}
\definecolor{amaranth}{rgb}{0.9, 0.17, 0.31}
\definecolor{auburn}{rgb}{0.43, 0.21, 0.1}
\definecolor{burntumber}{rgb}{0.54, 0.2, 0.14}

\usepackage{stfloats}
\usepackage{threeparttable}
\usepackage{subcaption}
\usepackage{comment}
\usepackage{xcolor}
\usepackage{authblk}
\usepackage{hyperref}

\newtheorem{prop}{Proposition}
\newtheorem*{prop*}{Proposition}
\newtheorem*{fact*}{Fact}
\newtheorem{assumption}{A.\hspace{-0.7mm}}

\begin{document}

\title{Decentralized Multi-Task Stochastic Optimization \\ With Compressed Communications}

\author[*]{Navjot Singh}
\author[$\dagger$]{Xuanyu Cao}
\author[*]{Suhas Diggavi}
\author[$\ddagger$]{Tamer Ba\c{s}ar}

\affil[*]{University of California, Los Angeles, USA}
\affil[*] {\text{navjotsingh@ucla.edu, suhas@ee.ucla.edu}\vspace{0.25cm}}
\affil[$\dagger$]{Hong Kong University of Science and Technology, Clear Water Bay, Hong Kong}
\affil[$\dagger$] {\text{eexcao@ust.hk}\vspace{0.25cm}}
\affil[$\ddagger$] {University of Illinois Urbana-Champaign, Urbana, USA}
\affil[$\ddagger$] {\text{basar1@illinois.edu}}

\date{\vspace{-5ex}}
\maketitle

\begin{abstract}
We consider a multi-agent network where each node has a stochastic (local) cost function that depends on the decision variable of that node and a random variable, and further  the decision variables of neighboring nodes are pairwise constrained. There is an aggregate objective function for the network, composed additively of the expected values of the local cost functions at the nodes, and the overall goal of the network is to obtain the minimizing solution to this aggregate objective function subject to all the pairwise constraints.  This is to be achieved at the nodes level using decentralized information and local computation, with exchanges of only compressed information allowed by neighboring nodes. The paper develops algorithms and obtains performance bounds for two different models of local information availability at the nodes: (i) sample feedback, where each node has direct access to samples of the local random variable to evaluate its local  cost, and (ii) bandit feedback, where samples of the random variables are not available, but only the values of the local cost functions at  two random points close to the decision are available to each node.  For both models, with compressed communication between neighbors, we have developed decentralized saddle-point algorithms that deliver performances no different  (in order sense) from  those without communication compression; specifically, we show that deviation from the global minimum value and violations of the constraints are upper-bounded by $O(T^{-{1\over 2}})$ and $O(T^{-{1\over 4}})$, respectively, where $T$ is the number of iterations. Numerical examples provided in the paper corroborate these bounds and demonstrate the communication efficiency of the proposed method.
\end{abstract}

\section{Introduction}
The emergence of multi-agent networks and the need to distribute computation across different nodes which have access to only piece of the network-wide data but are allowed to exchange information  under some resource constraints, have accelerated research efforts on decentralized and distributed optimization in multiple communities, particularly during the last 10-15 years. Spearheading this activity has been decentralized consensus optimization in static settings, where the goal is to minimize the sum of local cost functions, toward which \cite{NO09a} proposed a decentralized sub-gradient algorithm, whose convergence was further studied in \cite{YLY16}.  Following this initial work, several other consensus algorithms were introduced and studied, including alternating direction method of multipliers (ADMM) \cite{SLYWY14}, exact first-order algorithm \cite{SLWY15},  stochastic consensus optimization \cite{SVKB17, LLZ20}, and online consensus optimization with time-varying cost functions \cite{MC14, AGL17}. 

Consensus modeling framework requires, in essence, all nodes to converge to the same value, which however may not be appropriate in many network scenarios, where different nodes, even neighboring ones, may ultimately end up with different decision (or action) values. Such a scenario arises in, for example, distributed multitask adaptive signal processing, where the weight vectors at neighboring nodes are not the same \cite{CRS14, NRFS16}. One of the first papers that has analyzed such departure from consensus optimization is \cite{KSR17} where the formulation included proximity constraints between neighboring nodes, which were handled through construction of Lagrangians and using saddle-point algorithms, and extended to the asynchronous setting in \cite{BKR19}.

Decentralized algorithms are built on the assumption that there is some exchange of information among the nodes (at least among the neighboring nodes) which then propagates across the network toward achieving the global optimum in the limit. Extensive and frequent exchange of such information is generally practically impossible (due to bandwidth constraints on the edges of the underlying network which constitute the communication links, and computation and storage limitations, among many others), which inherently brings in a restriction on the amount and timing of the exchange of relevant current data. In the literature several studies have addressed these limitations through quantization of information or actions \cite{KBS07, EB16, BEO16, ELB16, ZC16}, by using only sign information on some differences \cite{ZYB19, CB21a}, by controlling the timing of transmissions through event triggering \cite{CB21b, SDGD21}, or by sparsification \cite{AHJKKR18,KSJ19, SCJ18}. Quantization in the context of decentralized optimization (and not consensus problems) has also been studied, with some of the algorithms leading to nonzero errors in convergence (see the early work \cite{RN05, NOOT08}) and others to exact convergence \cite{RMHP19}; see also \cite{AGLTV17} for quantized stochastic optimization. Some recent work has also used error-compensated compression in decentralized optimization, such as \cite{KSJ19, WHHZ18, SDGD20}.

Most of the existing works on decentralized optimization with quantized/compressed communications are, as discussed above,  focused on either {\it consensus} optimization or {\it unconstrained} optimization. Research departing from that trend was initiated in \cite{CB20}, which addressed the problem of  multitask learning (or distributed optimization with pairwise constraints) using quantized communications. More specifically, the model adopted  in that paper (with an underlying network topology) associated with each node a stochastic (local, individual) cost and with each pair of neighbors an inequality
constraint, e.g., proximity constraint.
Note that in such a formulation, different from consensus problems, each node has its own decision variable, but these cannot be picked independently because of the pairwise constraints. Further, the
distribution of the random variable in the stochastic local cost function of each node is unknown
and each node operates based on sequential feedback information,  rendering the formulation distinct from deterministic optimization. The paper developed stochastic saddle-point algorithms
with quantized communications between neighbors, and studied the impact of quantization on the optimization performance. One shortcoming of the result of \cite{CB20} is that the scheme developed led to nonzero convergence error; said differently, the algorithm in that work does not lead to converge to the exact optimal solution, as the number of iterations grow. This is precisely the issue we address in this paper, and achieve exact convergence, by employing a saddle-point algorithm along with an approach based  on error-compensated communication compression. Before further discussing the contents and contributions of this paper, let us point out that saddle-point algorithms (a.k.a. primal-dual algorithms) have been extensively used in the literature on constrained optimization,  such as deterministic  centralized optimization \cite{AHU58, NO09b}, decentralized optimization \cite{CNS14},
stochastic optimization \cite{KSR17, BKR19, EZCLR19}, and online optimization \cite{MJY12, CL19}. 

\subsection{Contributions}

In this paper, we address the problem of decentralized multi-agent stochastic optimization on a network, where each agent has a local stochastic convex cost function and each pair of neighbors is associated with an inequality constraint. The overall goal is to minimize the total (additive) expected cost of all agents subject to all the constraints on all edges, with all computation carried out at the nodes and with information exchanged among the nodes using compressed communication. 
We consider two scenarios of interest based on the sample information available locally at the nodes:
\begin{itemize}
	\item
	{\it Sample Feedback:} In this scenario, each node has access to the local samples of the random variable affecting its local cost function drawn from its distribution at any time instance during the optimization process
	and can thus evaluate its cost function and its gradient.
	\item
	{\it Bandit Feedback:} In this scenario, nodes do not have
	direct access to the samples, but rather can only observe
	values of the corresponding local cost functions at two points sufficiently close to
	the original node parameter. For references on bandit feedback in context of optimization (a.k.a. zeroth-order optimization), see \cite{FKM05, ADX10, DJWW15, S17, LKCTCA18, HHG19}.
\end{itemize}
Under both scenarios, 
the paper develops a decentralized saddle-point algorithm which leads to zero convergence error, even with a \emph{finite} number of bits for each iteration. Note that previous works in this topic \cite{CB20} required the number of bits to be unbounded for the error to diminish. 
	Specifically,  we show,  under some standard assumptions, that the expected suboptimality and the expected constraint violations are upper bounded by $O(T^{-{1\over 2}}) $ and $O(T^{-{1\over 4}})$, respectively, where $T$ is the number of iterations, despite using a finite number of bits. These bounds match, in order sense, the bounds for algorithms without communication compression. Hence, we get near optimal optimization performance even with finite number of bits under both scenarios. The paper also provides results of numerical experiments, which corroborate these bounds.

Accordingly, the main contributions of this paper are:
\begin{itemize}
	\item
	Using finite bit compressed sample feedback, with $T$ being the horizon of the optimization problem, achieving $O(1/\sqrt{T})$ closeness to optimum value of the objective function, and achieving  $O(T^{-{1\over 4}})$ constraint violation---both being the same as in the case without compression.
	\item
	Obtaining the same order bounds with bandit feedback, using only two-point feedback values.
\end{itemize}

\subsection{Paper Organization}

The rest of the paper is organized as follows: Section \ref{sec:prob_form} provides a precise formulation of the problem under consideration. Section \ref{sec:sample_feedback} develops the saddle-point algorithm (Algorithm \ref{algo}) under sample feedback, and provides convergence results and performance bounds (Theorem \ref{thm:opt}) along with essential points of the analyses and proofs, with some details relegated to an appendix (Appendix \ref{sec:app_feedback}). Section \ref{sec:bandit_feedback} presents the counterpart of Section \ref{sec:sample_feedback} for bandit feedback, with the corresponding algorithm (Algorithm \ref{algob}) and corresponding main result on convergence and performance bounds (Theorem \ref{thm:opt_bandit}), with some details relegated to Appendix \ref{sec:app_bandit}. Section \ref{sec:expts} discusses results of some numerical experiments.


\section{Problem Formulation} \label{sec:prob_form}
We first set up the multi-task decentralized optimization problem and establish the notation we use throughout.

We consider an undirected graph $\mathcal{G} = (\mathcal{V},\mathcal{E})$ with $|\mathcal{V}| = n $ participating with a total of $|\mathcal{E}| = m$ connected edges. As $\mathcal{G}$ is undirected, we assume that each connected node pair $(i,j) \in \mathcal{E}$ allows for bi-directional communication from $i$ to $j$ and from $j$ to $i$. The neighbor set of the node $i$ is denoted by $\mathcal{N}_i$.
 
Associated with each node $i \in [n]:= \{ 1,2,\hdots,n\}$ is an unknown data distribution which we denote by $\mathcal{P}_i$. The samples generated from the distribution are denoted by $\xi_i \sim \mathcal{P}_i$ where $\xi_i \in \Xi_i$. Each node also has a local cost function $f_i: \mathcal{X} \times \Xi_i \rightarrow \mathbb{R}^{+}$ which takes as input a sample $\xi_i \in \Xi_i$ and a local parameter $\bx_i \in \mathcal{X} \subset \mathbb{R}^d$ to yield the sample cost $f_i(\bx_i,\xi_i)$. Here, the set $\mathcal{X}$ corresponds to the set of feasible parameters the node can choose from, which is the same across all nodes. As an example, for supervised image recognition tasks, the sample $\xi_i$ for a node $i$ may correspond to an image-label pair with the set $\mathcal{X}$ being the set of all neural networks with a width of 2 layers and $\bx_i$ a particular 2-layer neural network. The local objective $f_i$ in this case may denote a cross-entropy loss function evaluated using the given image-label pair and the neural network. The expected cost for a node $i$ for parameter $\bx_i \in \mathcal{X}$ is denoted by $F_i(\bx_i) = \bbE_{\xi_i \sim \mathcal{P}_i} [ f_i(\bx_i,\xi_i) ]$. In general, we are interested in minimizing the expected cost for all the nodes $i\in [n]$. That is, we are interesting in finding node parameters $\{\bx_i\}_{i=1}^{n}$ that minimize the cost $F(\bx):= \sum_{i=1}^{n} F_i(\bx_i)$ where $F_i(\bx_i)$ denotes the expected cost of the node $i$ evaluated using parameter $\bx_i$ and $\bx \in \mathcal{X}^{n}$ denotes stacking of all the individual node parameters  $\{\bx_i\}_{i=1}^{n}$. Further, we assume that the node parameters are related via pairwise constraints on the connected nodes in the graph. Specifically, for any $i \in [n]$ and $j \in \mathcal{N}_i$, there is a function $g_{ij}: \mathcal{X} \times \mathcal{X} \rightarrow \mathbb{R}$ such that the inequality $g_{ij}(\bx_i,\bx_j) \leq 0$ should be satisfied. This may, for example, encode a proximity constraint on the node parameters by having $g_{ij}(\bx_i,\bx_j) = \Vert \bx_i - \bx_j \Vert_2^2 - c_{ij}$ where $\verts{.}_2$ denotes the $\ell_2$ norm and $c_{ij} \geq 0$ is a constant. In this paper, we assume that the constraint functions $g_{ij}(\bx_i,\bx_j)$ are symmetric in their parameters, i.e., $g_{ij}(\bx_i,\bx_j) = g_{ji}(\bx_j,\bx_i)$ for all $\bx_i,\bx_j \in \mathcal{X}$ and connected node pairs $(i,j) \in \mathcal{E}$, which leads to $m$ number of distinct pairwise constraints for all the parameters. With the notation now in place, we state the learning objective for the multi-task problem can be stated as follows:

\begin{align} \label{eq:main_obj}
	&\min_{\bx \in \mathcal{X}^n} F(\bx) = \sum_{i=1}^{n}\mathbb{E}_{\bxi_i}[f_i(\bx_i, \bxi_i) ] \\
	&\text{subject to} \quad g_{ij}(\bx_i,\bx_j) \leq 0, \quad \forall i \in [n], j \in \mathcal{N}_i \notag
\end{align}

To solve the problem in \eqref{eq:main_obj} in a decentralized manner, the nodes need to communicate during the optimization procedure which can be prohibitive for low bandwidth links or when the exchanged information updates among the nodes are large. To this end, in this paper we consider compression of the information exchanges among the nodes to make the communication efficient. In particular, we employ the notion of the compression operator proposed in \cite{SCJ18}:
\begin{definition} \label{def:comp}
	A (possibly randomized) function $\C: \mathbb{R}^d \rightarrow \mathbb{R}^d$ is called a compression operator, if there exists a positive constant $\omega \in (0,1)$, such that for every $\bx \in \mathbb{R}^d$:
	\begin{align} \label{eq:comp_op}
		\mathbb{E}\Vert \bx - \C(\bx) \Vert_2^2 \leq (1-\omega) \Vert \bx \Vert^2_2
	\end{align}
	where expectation is taken over the randomness of $\C$. We assume $\C(\mathbf{0}) = \mathbf{0}$.
\end{definition}
Many important sparsifiers and quantizers in the literature satisfy the above definition, with some of the major ones being:\\
{\sf(i)} $Top_k$ and $Rand_k$ sparsifiers \cite{SCJ18}, {\sf (ii)} Stochastic quantizer QSGD \cite{AGLTV17}, {\sf (iii)} The scaled $Sign$ quantizer \cite{KQSJ19}, and {\sf (iv)} composed quantization and sparsification operators in \cite{BDKD19}.

We consider two scenarios of interest based on the sampled information available locally at the nodes:
\begin{enumerate}[(i)]
	\item Sample Feedback: In this scenario we assume that each node $i$ has access to the local samples $\xi_i$ drawn from $\mathcal{P}_i$ at any time instance during the optimization procedure and can thus evaluate the cost function and its derivative.
	\item Bandit Feedback: In this scenario, nodes do not have a direct access to the samples, but rather can only observe values of the local cost function at two perturbations from the original node parameter. 
\end{enumerate}
We focus on these scenarios separately in Section \ref{sec:sample_feedback} and Section \ref{sec:bandit_feedback} respectively, where for each, we develop a compressed decentralized algorithm for optimizing \eqref{eq:main_obj} and present our theoretical convergence results for the developed schemes.

\section{Decentralized compressed optimization with Sample feedback} \label{sec:sample_feedback}
In this section we describe our approach for optimizing the objective in \eqref{eq:main_obj} for the case of sample feedback. In this setting, each node $i \in [n]$ has access to the sampled instance $\xi_i$ at any stage of the optimization procedure, and thus can evaluate the local objective $f_i(\bx_i,\xi_i)$ based on its local parameter $\bx_i$.

\subsection{Algorithm: with Sample Feedback}
We develop a stochastic saddle-point algorithm for solving \eqref{eq:main_obj} in a decentralized manner with compressed parameter exchanges. Our proposed scheme is presented in Algorithm~\ref{algo} and is based on finding a saddle point of the modified Lagrangian for the optimization problem in \eqref{eq:main_obj}. For a given sample $\xi_i$, we define this modified Lagrangian as follows:
\begin{align} \label{eq:sample_lagrangain}
	\mathcal{L}( \bx,\blambda ) = \sum_{i=1}^{n} \Big[f_i(\bx_i, \xi_i)  
	+ \sum_{j \in \mathcal{N}_i} \Big( \lambda_{ij} g_{ij}(\bx_i,\bx_j) - \frac{\delta \eta}{2} \lambda_{ij}^2\Big)  \Big]
\end{align}
On the L.H.S. of \eqref{eq:sample_lagrangain}, $\bx$ denotes the concatenation of all the model parameters $\{\bx_i\}_{i=1}^{n}$, each of which is in $\mathbb{R}^d$, leading to $\bx \in \mathbb{R}^{nd}$. For $i \in [n]$ and $j \in \mathcal{N}_i$, $\lambda_{ij} \geq 0$ denotes the Lagrangian multiplier associated with the constraint $g_{ij} (\bx_i, \bx_j) \leq 0$. Similarly, $\blambda$ on the L.H.S. denotes the concatenation of all $\lambda_{ij}$ for $i \in [n]$ and $j \in \mathcal{N}_i$, thus $\blambda \in \mathbb{R}^{m}$, where $m$ is twice the number of edges in the underlying undirected graph. The last term on the R.H.S. of \eqref{eq:sample_lagrangain} corresponds to a regularizer which mitigates the growth of the Lagrangian multiplier $\blambda$ during the saddle-point algorithm updates. In this term, $\eta > 0$ corresponds to the learning rate of the algorithm and $\delta >0$ is a control parameter.

To find the saddle point of the Lagrangian in \eqref{eq:sample_lagrangain}, we utilize alternating gradient updates of the primal variables concatenated in $\bx$, and the dual variables in $\blambda$. For any $i \in [n]$, the gradient of the modified Lagrangian with respect to (w.r.t.) the model parameter $\bx_i$ is given by:
\begin{align} \label{eq:sample_grad_primal}
	\nabla_{\bx_i}  \mathcal{L}(\bx,\blambda) &=  \sum_{j \in \mathcal{N}_i} \left[\lambda_{ij} \nabla_{\bx_i} g_{ij}(\bx_i,\bx_j) + \lambda_{ji} \nabla_{\bx_i} g_{ij}(\bx_j, \bx_i) \right]  + \nabla_{\bx_i} f_i(\bx_i, \xi_i) 
\end{align}
The gradient w.r.t. the Lagrangian multiplier $\lambda_{ij}$ for $i\in[n]$, $j \in \mathcal{N}_i$ is similarly given by:
\begin{align} \label{eq:sample_grad_dual}
	\frac{\partial}{\partial \lambda_{ij} } \mathcal{L}(\bx,\blambda) = g_{ij}(\bx_i,\bx_j) - \delta \eta \lambda_{ij}
\end{align}
The stochastic algorithm developed for updating the primal and dual variables via equations \eqref{eq:sample_grad_primal} and \eqref{eq:sample_grad_dual} is presented in Algorithm \ref{algo}, which is described below.

\begin{algorithm}[htp!]
	\caption{Compressed Decentralized Optimization with Sample Feedback}
	{\bf Initialize:} Random raw parameters $\btx_i^{(1)} \in \mathcal{X}$ individually, $\lambda_{ij}^{(1)} = 0$ for each $i \in [n]$, $j \in \mathcal{N}_i$, $\bhx_i^{(0)}=\mathbf{0}$ for each $i \in [n]$, number of iterations $T$, learning rate $\eta$, parameter $\delta >0$.  \\
	(Communicate in the first iteration without compression to ensure that $\btx^{(1)} = \bhx^{(1)}$ )
	\vspace{0cm}
	\begin{algorithmic}[1] \label{algo}
		\FOR{$t=1$ \textbf{to} $T$ in parallel for $i \in [n]$}
			\STATE Compute $\bq_i^{(t)} = \C ( \btx_i^{(t)} - \bhx_i^{(t-1)} ) $	
			\FOR{nodes $k \in \mathcal{N}_i \cup \{i\}$} 	
						\STATE Send $\bq_i^{(t)}$ and receive $\bq_k^{(t)}$
						\STATE Update $\bhx_k^{(t)} = \bhx_k^{(t-1)} + \bq_k^{(t)} $ 
						\STATE Compute $\bx_k^{(t)} = \projX (\bhx_k^{(t)}) $
			\ENDFOR
			\STATE Update running average for local parameter: \\
			$ \blx_{i,avg}^{(t)} = \frac{1}{t} \bx_i^{(t)} + \frac{t-1}{t} \blx_{i,avg}^{(t-1)} $
			\STATE Sample $\bxi_i^{(t)} \sim \mathcal{P}_i$ and compute $\nabla_{\bx_i} f_i(\bx_i^{(t)} ,\bxi_i^{(t)} )$
			\STATE For all $j \in \mathcal{N}_i$ compute $\nabla_{\bx_i} g_{ij}(\bx_i^{(t)}, \bx_j^{(t)})$
			\STATE Update the primal variable by gradient descent:
			\begin{align*}
				&\btx_i^{(t+1)}  = \projX \left( \btx_i^{(t)}  - \eta  \nabla_{\bx_i} f_i(\bx_i^{(t)} ,\bxi_i^{(t)} )  - 2 \eta \displaystyle \sum_{j \in \mathcal{N}_i} \lambda_{ij}^{(t)} \nabla_{\bx_i} g_{ij}(\bx_i^{(t)}, \bx_j^{(t)} )    \right)
			\end{align*}
			\STATE For $j \in \mathcal{N}_i$, update the dual variables through gradient ascent:
			\begin{align*}
				& \lambda_{ij}^{(t+1)} = \left[ \lambda_{ij}^{(t)} + \eta \left( g_{ij}(\bx_i^{(t)} , \bx_j^{(t)} ) - \delta \eta \lambda_{ij}^{(t)}  \right)  \right]^{+} 
			\end{align*}
		\ENDFOR
	\end{algorithmic}
	{\bf Output:} Time averaged parameters $\blx_{i,avg}^{(T)}$ for all $i \in [n]$.
\end{algorithm}

Our proposed scheme in Algorithm \ref{algo} is a stochastic saddle-point algorithm to minimize the objective \eqref{eq:main_obj} by finding a saddle point of the modified Lagrangian in \eqref{eq:sample_lagrangain} in a communication efficient manner. In particular, each node is allowed to exchange with its neighboring nodes only compressed parameters, via the compression operator in \eqref{eq:comp_op}. To realize exchange of compressed parameters between workers, for each node $i \in [n]$ and its associated \emph{raw} parameter $\btx_i$, all nodes $j \in \mathcal{N}_i$ maintain an estimate $\bhx_i$ of $\btx_i$, so, each node $i \in [n]$ has access to $\bhx_j$ for all $j \in \mathcal{N}_i$. The parameter $\btx_i$ is called raw as it corresponds to the model parameter before any compression in our algorithm. We refer to $\bhx_i$ as the \emph{copy} parameter of the node $i$.

We first initialize the regularization parameter $\delta$ (see Theorem \ref{thm:opt} for definition) and learning rate $\eta$. We initialize the parameter copies of all the nodes as $\bhx_i = \mathbf{0}$ for all $i \in [n]$ and allow each node to communicate with its neighbors in the first round without any compression. This is to ensure that $\btx_i^{(1)} = \bhx_i^{(1)}$ for all the nodes (this is a requirement to control the error encountered via compression, c.f. Lemma \ref{lem:e_t}). At any time step $t \in [T]$ of the algorithm, node $i$ first computes the compressed update to its copy parameter, given by $\bq_i^{(t)}$ (line 2) and then sends and receives these copy parameter updates from its neighbor nodes in $\mathcal{N}_i$ (line 3). Importantly, these copy parameter updates are compressed using the operator in \eqref{eq:comp_op}, and thus the communication is efficient. After receiving the copy updates from its neighbors, each node updates the locally available copy parameters of its neighbors and its own copy parameters (line 5) and ensures that the copy parameters lie in the set $\mathcal{X}$ by taking a projection to form the local node parameter $\bx_i^{(t)}$ (line 6). As the node $i$ has access to the updated copy parameters of its neighbors, it also has access to $\bx_j^{(t)}$ for all $j \in \mathcal{N}_i$. With the local node parameter evaluated, the node can now update its running average of parameters (line 8). 

For the stochastic saddle-point update steps at time $t$, in the case of sample feedback, the node can sample datapoint $\bxi_i^{(t)}$ and evaluate the gradient using the previously computed node parameter $\bx_i^{(t)}$ (line 9). Since node $i$ also has access to the parameters $\bx_j^{(t)}$ for its neighbors $j \in \mathcal{N}_i$, it can compute the gradient w.r.t. the pairwise constraint function $g_{ij}$ evaluated at $\bx_i^{(t)},\bx_j^{(t)}$ (line 10). Thus, the node can evaluate the gradient of the modified Lagrangian w.r.t. the primal local node parameters as in \eqref{eq:sample_grad_primal} and take a gradient descent step to update the raw node parameter $\btx_i^{(t)}$. Similarly, the dual variables $\lambda_{ij}^{(t)}$ are also updated via a gradient ascent step (line 12) similar to the update in \eqref{eq:sample_grad_dual} and then projected on the positive real space.

\paragraph{Symmetry of dual updates:}
Note that the derived expression for the gradient $\nabla_{\bx_i}  \mathcal{L}(\bx,\blambda)$, consists of the dual parameters $\lambda_{ij}$ and $\lambda_{ji}$. Meanwhile, the update in line 11 of Algorithm~\ref{algo} considers these parameters to be the same for all time $t \in [T]$. We describe the reasoning behind this update in the following induction argument. The dual variables are initialized to 0, that is, $\lambda_{ij}^{(1)} = 0$ for all $i \in [n]$ and $j \in \mathcal{N}_i$. Thus for any connected nodes $i,j$, for $t=1$, the condition $\lambda_{ij}^{(t)} = \lambda_{ji}^{(t)}$ holds. Next, we assume that for any arbitrary $\tau \in [T]$, $\tau \neq 1$, it is the case that $\lambda_{ij}^{(\tau)} = \lambda_{ji}^{(\tau)}$. Thus for the time step $t = \tau + 1$, by the update given in line 12 of Algorithm \ref{algo}, we have:
\begin{align*}
	\lambda_{ij}^{(\tau+1)} &= \left[ \lambda_{ij}^{(\tau)} + \eta \left( g_{ij}(\bx_i^{(\tau)} , \bx_j^{(\tau)} ) - \delta \eta \lambda_{ij}^{(\tau)}  \right)  \right]^{+} 
	&\stackrel{(a)}{=} \left[ \lambda_{ji}^{(\tau)} + \eta \left( g_{ji}(\bx_j^{(\tau)} , \bx_i^{(\tau)} ) - \delta \eta \lambda_{ji}^{(\tau)}  \right)  \right]^{+} 
	&= \lambda_{ji}^{(\tau+1)}
\end{align*}
where (a) follows from the fact that $\lambda_{ij}^{(\tau)} = \lambda_{ji}^{(\tau)}$ and the symmetry of the pairwise constraints $g_{ij}$ for connected nodes $i,j$. Thus, as the induction step holds for arbitrary $\tau \in [T]$ and for the base case $t =1$, it follows that $\lambda_{ij}^{(t)} = \lambda_{ji}^{(t)}$ for all $t \in [T]$ for all $i \in[n], j\in \mathcal{N}_i$.

\paragraph{Justification for raw parameter updates:}
Note that in the steps given in lines (9-11) in Algorithm \ref{algo}, the gradients are evaluated at the node parameters $\{\bx_i^{(t)}\}_{i=1}^{n}$, while the updates are made to the raw parameters $\{\btx_i^{(t)}\}_{i=1}^{n}$ via gradient descent. The reason for this is that in our scheme, the raw parameters effectively play the role of a \emph{virtual} parameter, which mimic SGD-like updates (c.f. line 11), with the gradients evaluated at a different (perturbed) parameter. The notion of such virtual parameters to analyze convergence has been promising lately in stochastic optimization within the perturbed iterate analysis framework, see \cite{MPPRRJ17,SDGD21, SCJ18, BDKD19}. The key idea to analyze convergence in such settings is to control the difference of the iterates $\verts{\bx_i^{(t)} - \btx_i^{(t)}}_2$ for all $i \in [n]$. Controlling this difference is one key contribution of our work, c.f. Lemma \ref{lem:e_t}.


\subsection{Main Result: Sample Feedback}
\allowdisplaybreaks
We now present our theoretical result on the convergence rate of Algorithm \ref{algo} for decentralized optimization for the case with sample feedback. We first present and discuss the set of assumptions our result is based on.

\begin{assumption}\label{assump:X}
		The set of admissible model parameters $\mathcal{X}$, is closed, convex and bounded, i.e., there exists a constant $R>0$ such that $\verts{\btx}_2 \leq \frac{R}{\sqrt{n}} $, for all  $\btx \in  \cX$ 
		
	\end{assumption}
	
	\begin{assumption}\label{assump:convexity}
		For any $i \in [n]$, the local objective $f_i(\bx_i,\bxi_i)$ is convex in $\bx_i$ for any $\bxi_i \in \Xi_i$. The pairwise constraint function $g_{ij}(\bx_i,\bx_j)$ is (jointly) convex in $\bx_i$ and $\bx_j$, for any pair $i\in[n],j\in\cN_i$.
	\end{assumption}
	
	\begin{assumption}\label{assump:gradient}
		For any $i\in[n]$, there exists a constant $G_i>0$ such that $\forall \bx_i\in \cX$:
		\begin{align} \label{assump:gradient_func}
			\bbE_{\bxi_i\sim \cP_i} \left[\| \nabla_{\bx_i}f_i(\bx_i,\bxi_i) \|_2^2\right]\leq G_i^2.
		\end{align}
		To simplify the notation, we also define $G := \sqrt{\sum_{i=1}^{n} G_i^2 } $.	
		Additionally, for any $i\in[n],j\in \cN_i$, we assume that there exists a constant $G_{ij}>0$ such that $\forall \bx_i,\bx_j\in \cX$:
		\begin{align} \label{assump:gradient_constraint}
			\left\|\left[\nabla_{\bx_i}g_{ij}(\bx_i,\bx_j)^{\trans},\nabla_{\bx_j}g_{ij}(\bx_i,\bx_j)^{\trans}\right]^{\trans}\right\|_2 \leq G_{ij}.
		\end{align}
		We define $\tilde{G}  := \max_{i \in [n], j \in \mathcal{N}_i } G_{ij}$
	\end{assumption}
	
	\begin{assumption}\label{assump:g}
		For any $i\in[n],j\in \cN_i$, the pairwise constraint function $g_{ij}$ is bounded. That is, there exists a constant $C_{ij}>0$ such that $|g_{ij}(\bx_i,\bx_j)|\leq C_{ij}$, $\forall\bx_i,\bx_j\in \cX$. We further define $C^2 :=  \sqrt{ \sum_{i=1}^{n}  \sum_{j \in \mathcal{N}_i}  C^2_{ij} }$ to simplify the results presented below.
\end{assumption}

We remark that Assumptions A.\ref{assump:X}-A.\ref{assump:g} are frequently used in convergence rate analysis of convex optimization algorithms, even without compression. The assumption on the boundedness of the parameter space $\mathcal{X}$ and the boundedness of the constraint functions have been made earlier in \cite{CB20,CLG17}. The assumption on boundedness of the gradient of the objectives (equation \eqref{assump:gradient_func}) has also been made earlier in \cite{CB20, NO09a, CLG17} and boundedness of gradients of the constraint functions (equation \eqref{assump:gradient_constraint}) have been assumed in \cite{CB20,YN17,NN13}\footnote{The bounded gradient assumption for compressed decentralized optimization has been relaxed in one of our previous works \cite{SDGD21}. We remark that the argument for relaxing this assumption can similarly be extended to the analysis in this paper, a technicality which we omit, in interest of keeping the analysis relatively cleaner, and to focus on the main novelty of analyzing compressed communication in the pairwise multi-task setting.}.

With these assumptions in place, we now present our main theoretical result in Theorem \ref{thm:opt} below for the convergence rate of Algorithm \ref{algo}.
The result is stated in terms of the stacked vector $\bx$, which corresponds to the concatenation of the parameters $\{\bx_i\}_{i=1}^{n}$, and thus is $n \times d$ dimensional. The vector $\bx^{*}$ represents the stacked optimal parameters which is the solution of the optimization problem \eqref{eq:main_obj}.
 The proof details for Theorem~\ref{thm:opt} are presented in Section \ref{subsec:proof_sample}.
\begin{theorem} \label{thm:opt}
	Consider running Algorithm \ref{algo} for $T$ iterations with fixed step size $\eta = \frac{a}{\sqrt{T}}$ for positive constant $a$ and regularization parameter $\delta$ taking values in $\left[ \frac{1 - \sqrt{1 - \frac{64 \eta^2 (1+m)\tilde{G}^2}{\omega^2} }}{4 \eta^2}, \frac{1 + \sqrt{1 - \frac{64 \eta^2 (1+m)\tilde{G}^2}{\omega^2}}}{4 \eta^2} \right]$  where $\omega \in (0,1)$ is the compression factor. Then, under assumptions A.\ref{assump:X} - A.\ref{assump:g}, for $T \geq \frac{64a^2 (1+m)\tilde{G}^2}{\omega^2}$, the expected value of $F$ evaluated at the stacked time-averaged vector $\blx^{(T)}_{avg} := \frac{1}{T} \sum_{t=1}^{T}\bx^{(t)} $ satisfies:
	\begin{align} \label{eq:thm_1}
		\mathbb{E}[F(\blx^{(T)}_{avg})] {-}   F(\bx^{*})  \leq \frac{2R^2}{a \sqrt{T}} + \frac{a}{\sqrt{T}} \left( \frac{4}{\omega^2}(1{+}m)G^2 {+} 2C^2 \right) 
	\end{align} 
	For any $i\in [n]$, $j \in \mathcal{N}_i$, the constraint function $g_{ij}$ satisfies:
	\begin{align} \label{eq:thm_2}
		\mathbb{E} \left[ g_{ij}(\blx_{i,avg}^{(T)} , \blx_{j,avg}^{(T)})  \right]  
		& \leq \frac{1}{T^{\frac{1}{4}}} \left( \sqrt{\frac{8GR}{a}} + \sqrt{8\delta aGR}  \right)   +  \frac{1}{\sqrt{T}} \sqrt{2\left(  2R^2\delta + \frac{4}{\omega^2}(1+m)G^2 + C^2  \right)} \notag \\
		& \qquad  + \frac{1}{\sqrt{T}} \sqrt{2 \delta a^2  \left( \frac{4}{\omega^2} (1+m)G^2 + C^2 \right)} + \frac{2R}{a \sqrt{T}}
	\end{align}
	where the $d$-dimensional vector $\blx_{k,avg}^{(T)}$ denotes the time averaged parameter for node $k \in [n]$ in $\blx_{avg}^{(T)}$.
\end{theorem}

Theorem \ref{thm:opt} establishes that for any given compression requirement $\omega \in (0,1)$, the suboptimality of the objective, $\bbE [F(\blx^{(T)})] - F(\bx^{*})$, is $\mathcal{O}\left(\frac{1}{T^{\nicefrac{1}{2}}}\right)$, and the expected constraint violation $\mathbb{E} [ g_{ij}(\blx_{i}^{(T)} , \blx_{j}^{(T)})  ]$ for any connected node pair $(i,j)$ is $\mathcal{O}\left(\frac{1}{T^{\nicefrac{1}{4}}}\right)$. Thus, the difference between the attained objective and the global minimum of \eqref{eq:main_obj}, as well as the constraint violations can be made arbitrarily small by increasing the number of iterations the algorithm is run for.

\subsection{Convergence analysis} \label{subsec:proof_sample}
We now present the main proof details of Theorem \ref{thm:opt} which establishes a convergence rate of Algorithm \ref{algo}.

\subsubsection{Preliminaries} \label{subsubsec:sample_prelim}
We first introduce a compact vector notation which we will use throughout the proof. Consider the stacked (concatenated) vector of the node parameter vectors $\{\bx_i\}_{i=1}^{n}$ which we denote by $\bx$, and thus is $nd$-dimensional. Similarly, we define the vector $\blambda$ of size $m$ which stacks together the dual variables $\lambda_{ij}$ for $i \in [n]$ and $j \in \mathcal{N}_i$. The vector $\bg(\bx)$ represents the the stacked vector of constraint values $g_{ij}(\bx_j,\bx_j)$, and is also $m$-dimensional. Finally, $\bxi$ denotes the concatenated vector of samples across the nodes. The projection $\projXn(\bx)$ refers to projection of $\bx$ on the space $\mathcal{X}^n$ where each individual node parameter comprising $\bx$ is projected onto $\mathcal{X}$. Under this compact notation, the modified Lagrangian presented in \eqref{eq:sample_lagrangain} can be re-written as:
\begin{align} \label{eq:lagrangian}
	\mathcal{L}(\bx,\blambda) = f(\bx,\bxi)+ \blambda^T\bg(\bx) - \frac{\delta \eta}{2}\Vert \blambda \Vert^2
\end{align}
We now present a few auxiliary results which we use through the course of the proof. Some of these can be derived from the assumptions made in A.\ref{assump:X}-A.\ref{assump:g}.

\begin{fact} \label{fact:proj_bound}
	Suppose $\mathcal{A} \subset \mathbb{R}^l$ is closed and convex. Then, for any $\by \in \mathbb{R}^l$ and $\bx \in \mathcal{A}$, we have:
	\begin{align*}
		\verts{\bx - \Pi_{\mathcal{A}}(\by)}_2 \leq \verts{\bx - \by}_2 
	\end{align*}
	where $\Pi_{\mathcal{A}}(\by)$ denotes the projection of $\by$ on the set $\mathcal{A}$. 
\end{fact}

\begin{fact} \label{fact:grad_lagrangian_bound}
	(Bound on gradients of the Lagrangian)		
	Consider the Lagrangian function over the primal and dual variables defined in \eqref{eq:lagrangian}. Then we have the following bounds:		
	\begin{enumerate} [(a)]
		\item $\bbE \Vert \nabla_{\blambda} \mathcal{L}(\bx^{(t)},\blambda^{(t)}) \Vert^2 \leq 2 C^2 + 2 \delta^2 \eta^2 \mathbb{E} \Vert \blambda^{(t)} \Vert^2$
		\item $\mathbb{E} \verts{\nabla_{\bx} \mathcal{L}(\bx^{(t)},\blambda^{(t)})}^2 \leq (1+m) \left( G^2  + 	\widetilde{G}^2 \mathbb{E} \verts{\blambda}^2  \right)$
	\end{enumerate}
	where $C^2, \tilde{G} $  and $G $ are as defined in Assumption \ref{assump:gradient} and \ref{assump:g}. We provide the proof for this fact in Appendix \ref{sec:app_feedback}.
\end{fact}

\begin{fact} \label{fact:F_diff_bound}
	For all $\bx \in \mathcal{X}^n,$ we have:
	\begin{align*}
		\mathbb{E}[F(\bx)]-F(\bx^{*}) > -4GR
	\end{align*}
	where $\bx^{*}$ is an optimal solution of \eqref{eq:main_obj}, and $R,G$ are as defined in Assumptions \ref{assump:X} and \ref{assump:g}, respectively. We provide a proof for Fact \ref{fact:F_diff_bound} in Appendix \ref{sec:app_feedback}.
\end{fact}

%

\subsubsection{Proof of Theorem \ref{thm:opt}}
We first consider the following lemma which establishes a relationship between the Lagrangian function and the primal, dual variables in Algorithm~\ref{algo}. The proof for the lemma, provided in Appendix~\ref{sec:app_feedback}, relies on considering the update steps of the primal and dual variables in Algorithm~\ref{algo} and invoking convexity/concavity arguments for the Lagrangian function.
\begin{lemma} \label{lem:proof_sample_init}
	Consider the update steps in Algorithm~\ref{algo} with learning rate $\eta$ and parameter $\delta \geq 0$. Under assumptions~A.\ref{assump:X}-A.\ref{assump:g}, for any $\bx \in \mathcal{X}^{n}$ and $\blambda \in \mathbb{R}^m$ with $\blambda \succeq \mathbf{0}$, the summation of the Lagrangian function satisfies:
	\begin{align*} 
		\sum_{t=1}^{T} \mathbb{E}\left(\mathcal{L}(\bx^{(t)}, \blambda) - \mathcal{L}(\bx, \blambda^{(t)})\right)  & \leq \frac{1}{2 \eta} \left(  \verts{\blambda }^2 + 4R^2 \right)  + \eta T \left( (1+m)G^2 + C^2 \right) + \frac{1}{2\eta} \sum_{t=1}^{T}    \bbE \verts{ \bx^{(t)} - \tilde{\bx}^{(t)} }^2  \notag \\
		& \quad + \eta \left( (1+m)\tilde{G}^2 + \delta^2\eta^2 \right)  \sum_{t=1}^{T} \mathbb{E}[ \Vert\blambda^{(t)}\Vert^2  ]
	\end{align*}
	where $G,C, \tilde{G}$, $R$ are defined in assumptions~A.\ref{assump:X}-A.\ref{assump:g}.
\end{lemma}
We first note that using definition of Lagrangian from \eqref{eq:lagrangian} and $\bbE [f(\bx^{(t)} , \bxi^{(t)})] = F(\bx^{(t)})$, the L.H.S. of the result in Lemma~\ref{lem:proof_sample_init} can also be written as following for any $\blambda \succeq \mathbf{0}$:
\begin{align*}
	\mathbb{E} \left[ \sum_{t=1}^{T} \left(\mathcal{L}(\bx^{(t)}, \blambda) - \mathcal{L}(\bx, \blambda^{(t)})\right) \right] 
	& =  \sum_{t=1}^{T} ( \mathbb{E}[F(\bx^{(t)})] {-}   F(\bx^{*}) ) {+} \left\langle \blambda, \sum_{t=1}^T \mathbb{E}[\bg(\bx^{(t)})] \right\rangle {-} \frac{\delta \eta T}{2} \Vert \blambda \Vert^2   \\
	& \qquad  - \mathbb{E} \left[ \sum_{t=1}^{T} \langle \blambda^{(t)}, \bg(\bx^{*}) \rangle \right] + \frac{\delta \eta}{2} \mathbb{E}  \left[ \sum_{t=1}^{T} \Vert \blambda^{(t)} \Vert^2\right] 
\end{align*}
Rearranging the terms and employing the bound from Lemma~\ref{lem:proof_sample_init}, for any $\blambda \succeq \mathbf{0}$, we thus have:
\begin{align} \label{eq:proof_11}
	&\sum_{t=1}^{T} \left( \mathbb{E}[F(\bx^{(t)})] {-}   F(\bx^{*}) \right) + \left\langle \blambda, \sum_{t=1}^T \mathbb{E}[\bg(\bx^{(t)})] \right\rangle - \frac{\delta \eta T}{2} \Vert \blambda \Vert^2 
	 - \mathbb{E} \left[ \sum_{t=1}^{T} \langle \blambda^{(t)}, \bg(\bx^{*}) \rangle \right]  \notag \\
	& \leq \frac{1}{2 \eta} \left(  \verts{\blambda }^2 {+} 4R^2 \right) {+} \frac{1}{2\eta} \sum_{t=1}^{T}    \mathbb{E} \verts{ \be^{(t)} }^2   {+} \eta T \left( (1{+}m)G^2 {+} C^2 \right) + \eta \left( (1+m)\tilde{G}^2 + \delta^2\eta^2 - \frac{\delta}{2} \right)  \sum_{t=1}^{T} \mathbb{E}[ \Vert\blambda^{(t)}\Vert^2  ]
\end{align}
where we have defined $\bbE \verts{\be^{(t)}}^2 := \bbE \verts{\btx^{(t)} - \bx^{(t)}}^2 $ on the R.H.S. of \eqref{eq:proof_11}. This term relates to the error between the copies of the parameters at time $t$ (denoted by $\btx^{(t)}$) and the true parameters of the nodes (given by $\bx^{(t)}$). We provide a bound for this term in Lemma \ref{lem:e_t} stated below, the proof of which is provided in Section \ref{sec:app_feedback}.

\begin{lemma} \label{lem:e_t}
	For the update steps in Algorithm \ref{algo}, the norm of expected error $\mathbb{E}\Vert \be^{(t)} \Vert$ for any $t \in [T]$ is bounded as:
	\begin{align*}
		\mathbb{E}\Vert \be^{(t)} \Vert^2 \leq \frac{2 \eta^2}{\omega} \sum_{k=0}^{t-2} \left(  1 {-} \frac{\omega}{2} \right)^k \mathbb{E}\Vert\nabla_{\bx}\mathcal{L}_{t{-}1{-}k}(\bx^{(t{-}1{-}k)},\blambda^{(t{-}1{-}k)} ) \Vert^2
	\end{align*}
	where $\eta$ is the learning rate and $\omega$ is the compression factor.
\end{lemma}
Plugging the bound for $\mathbb{E} \Vert \be^{(t)}\Vert^2$ from Lemma \ref{lem:e_t} into \eqref{eq:proof_11}:
\begin{align} \label{eq:temp_double_sum1} 
	&\sum_{t=1}^{T} \left( \mathbb{E}[F(\bx^{(t)})] {-}   F(\bx^{*}) \right) + \left\langle \blambda, \sum_{t=1}^T \mathbb{E}[\bg(\bx^{(t)})] \right\rangle - \frac{\delta \eta T}{2} \Vert \blambda \Vert^2 - \mathbb{E} \left[ \sum_{t=1}^{T} \langle \blambda^{(t)}, \bg(\bx^{*}) \rangle \right]  \notag \\
	&  \leq \frac{1}{2 \eta} \left(  \verts{\blambda }^2 + 4R^2 \right) + \eta T \left( (1+m)G^2 + C^2 \right)  + \frac{\eta}{\omega} \sum_{t=1}^{T}  \sum_{k=0}^{t-2} \left(  1 {-} \frac{\omega}{2} \right)^k \mathbb{E}\verts{\nabla_{\bx}\mathcal{L}_{t{-}1{-}k}(\bx^{(t{-}1{-}k)},\blambda^{(t{-}1{-}k)} )}^2 \notag \\
	&  \quad + \eta \left( (1+m)\tilde{G}^2 + \delta^2\eta^2 - \frac{\delta}{2} \right)  \sum_{t=1}^{T} \mathbb{E}[ \Vert\blambda^{(t)}\Vert^2  ] \\ 		
	& = \frac{1}{2 \eta} \left(  \verts{\blambda }^2 + 4R^2 \right) + \eta T \left( (1+m)G^2 + C^2 \right)   + \eta \left( (1+m)\tilde{G}^2 + \delta^2\eta^2 - \frac{\delta}{2} \right)  \sum_{t=1}^{T} \mathbb{E}[ \Vert\blambda^{(t)}\Vert^2  ] \notag \\	
	& \qquad + \frac{\eta}{\omega} \sum_{k=1}^{T-1}  \sum_{t=k+1}^{T} \left(  1 - \frac{\omega}{2} \right)^{(t-1-k)} \mathbb{E}\verts{\nabla_{\bx}\mathcal{L}_{k}(\bx^{(k)},\blambda^{(k)} )}^2 \notag 
\end{align}
where the equality follows from rewriting the double-sum of the second term.
Using $\sum_{t=k{+}1}^{T} \left(  1 {-} \frac{\omega}{2} \right)^{(t{-}1{-}k)} \leq \sum_{t=0}^{\infty} \left(  1 {-} \frac{\omega}{2} \right)^{(t)} = \frac{2}{\omega}$, we get:
\begin{align} \label{eq:temp_double_sum2}
	&\sum_{t=1}^{T} ( \mathbb{E}[F(\bx^{(t)})] {-}   F(\bx^{*}) ) + \left\langle \blambda, \sum_{t=1}^T \mathbb{E}[\bg(\bx^{(t)})] \right\rangle {-} \frac{\delta \eta T}{2} \Vert \blambda \Vert^2   - \mathbb{E} \left[ \sum_{t=1}^{T} \langle \blambda^{(t)}, \bg(\bx^{*}) \rangle \right]  \notag \\
	& \leq \frac{1}{2 \eta} \left(  \verts{\blambda }^2 + 4R^2 \right) + \eta T \left( (1+m)G^2 + C^2 \right) + \eta \left( (1+m)\tilde{G}^2 + \delta^2\eta^2 - \frac{\delta}{2} \right)  \sum_{t=1}^{T} \mathbb{E}[ \Vert\blambda^{(t)}\Vert^2  ] \notag \\	
	& \qquad + \frac{2\eta}{\omega^2} \sum_{t=1}^{T-1} \mathbb{E}\verts{\nabla_{\bx}\mathcal{L}_{t}(\bx^{(t)},\blambda^{(t)} )}^2
\end{align}
Using the bound from (b) in Fact \ref{fact:grad_lagrangian_bound} for the last term in above, and noting that $\frac{2}{\omega^2} > 1$ gives us:
\begin{align*}
	&\sum_{t=1}^{T} ( \mathbb{E}[F(\bx^{(t)})] {-}   F(\bx^{*}) ) + \left\langle \blambda, \sum_{t=1}^T \mathbb{E}[\bg(\bx^{(t)})] \right\rangle {-} \frac{\delta \eta T}{2} \Vert \blambda \Vert^2  - \mathbb{E} \left[ \sum_{t=1}^{T} \langle \blambda^{(t)}, \bg(\bx^{*}) \rangle \right]  \notag \\
	& \leq \frac{1}{2 \eta} \left(  \verts{\blambda }^2 + 4R^2 \right) + \eta T \left( \frac{4}{\omega^2}  (1+m)G^2 + C^2 \right)  + \eta \left( \frac{4}{\omega^2}(1+m)\tilde{G}^2 + \delta^2\eta^2 - \frac{\delta}{2} \right)  \sum_{t=1}^{T} \mathbb{E}[ \Vert\blambda^{(t)}\Vert^2  ]
\end{align*}

We now focus on the last term in the above equation, which has a coefficient of $\left( \frac{4}{\omega^2}(1+m)\tilde{G}^2 + \delta^2\eta^2 - \frac{\delta}{2} \right) $. To get rid of the last term in the upper bound, we choose the value of $\delta$ such that this coefficient is negative. In essence, this requires choosing a value of $\delta$ such that $2\eta^2 \omega^2 \delta^2  - \omega^2 \delta +  8(1+m)\tilde{G}^2   \leq 0$. Since this is a quadratic in $\delta$, it can be easily checked that the following range of $\delta$ satisfies the desired inequality:
\begin{align*}
	 \delta \in \left[ \frac{1 - \sqrt{1 - \frac{64 \eta^2 (1+m)\tilde{G}^2}{\omega^2} }}{4 \eta^2}, \frac{1 + \sqrt{1 - \frac{64 \eta^2 (1+m)\tilde{G}^2}{\omega^2}}}{4 \eta^2} \right]
\end{align*}
To ensure that the values of the interval bounds are real, we require running the algorithm for $T \geq \frac{64a^2 (1+m)\tilde{G}^2}{\omega^2}$.
We also remark that for $T \rightarrow \infty$ (i.e., $\eta \rightarrow 0$ for the choice $\eta = \frac{a}{\sqrt{T}}$), the left interval value of $\delta$ converges to a positive constant, meaning that we can still choose a finite value of $\delta$ in the infinitesimally small learning rate regime so that our bounds derived below are not vacuous.
Choosing $\delta$ in the stated range and using the fact $\mathbb{E} \left[ \sum_{t=1}^{T} \langle \blambda^{(t)}, \bg(\bx^{*}) \rangle \right] \leq 0$ since $\blambda^{(t)} \succeq \mathbf{0} $ for $t \in [T]$ and $\bg(\bx^{*}) \preceq \mathbf{0}$ and rearranging the terms, we get:
\begin{align} \label{eq:proof_7}
	&\sum_{t=1}^{T} \left( \mathbb{E}[F(\bx^{(t)})] -   F(\bx^{*}) \right) + \left\langle \blambda, \sum_{t=1}^T \mathbb{E}[\bg(\bx^{(t)})] \right\rangle  - \left(\frac{\delta \eta T}{2} + \frac{1}{2 \eta} \right) \Vert \blambda \Vert^2  \leq \frac{2R^2}{\eta} + \eta T \left( \frac{4}{\omega^2}  (1+m)G^2 + C^2 \right)
\end{align}


Recall that $\blambda$ can be any non-negative vector. We set it as
$\blambda = \frac{\left[ \mathbb{E} \left[ \sum_{t=1}^{T} \bg(\bx^{(t)})  \right]  \right]^{+}}{\delta \eta T + \frac{1}{\eta}}$.
Plugging this choice into \eqref{eq:proof_7} gives us:
\begin{align} \label{eq:proof_9}
	\sum_{t=1}^{T} \left( \mathbb{E}[F(\bx^{(t)})] -   F(\bx^{*}) \right) +  \sum_{i=1}^{n} \sum_{j \in \mathcal{N}_i} \frac{\left( \left[\mathbb{E} \left[ \sum_{t=1}^{T} g_{ij}(\bx_i^{(t)} , \bx_j^{(t)})  \right] \right]^{+} \right)^2}{2\left( \delta \eta  T + \frac{1}{ \eta} \right)} \leq \frac{2R^2}{\eta} + \eta T \left( \frac{4}{\omega^2}  (1+m)G^2 + C^2 \right)
\end{align}
Dividing both sides of \eqref{eq:proof_9} by $T$ and noting that the second term on the L.H.S. of \eqref{eq:proof_9} is positive, we can bound the time-average sub-optimality of $F$ as:
\begin{align*}
	 \sum_{t=1}^{T} \frac{\left( \mathbb{E}[F(\bx^{(t)})] -   F(\bx^{*}) \right)}{T}  \leq \frac{2R^2}{\eta T}  + \eta \left( \frac{4}{\omega^2}  (1{+}m)G^2 {+} C^2 \right)
\end{align*} 
Using the convexity of $F$ and setting $\eta = \frac{a}{\sqrt{T}}$ for some positive constant $a$, for $\blx^{(T)}_{avg} := \frac{1}{T} \sum_{t=1}^{T} \bx^{(t)} $ we have:
\begin{align} \label{eq:opt_result}
 \mathbb{E}[F(\blx^{(T)}_{avg})] {-}   F(\bx^{*})  \leq \frac{2R^2}{a \sqrt{T}} {+} \frac{a}{\sqrt{T}} \left( \frac{4}{\omega^2}(1{+}m)G^2 {+} 2C^2 \right) 
\end{align}
This concludes the proof of the convergence rate for the objective suboptimality given in \eqref{eq:thm_1} in Theorem \ref{thm:opt}. \\
We now prove our result for the pairwise constraint functions. From Fact \ref{fact:F_diff_bound}, $\forall \bx \in \mathcal{X}^n,$ we have $\mathbb{E}[F(\bx)]-F(\bx^{*}) > -4GR$. Using this inequality in \eqref{eq:proof_9} yields:
\begin{align*}
	& \sum_{i=1}^{n} \sum_{j \in \mathcal{N}_i} \left( \left[\mathbb{E} \left[ \sum_{t=1}^{T} g_{ij}(\bx_i^{(t)} , \bx_j^{(t)})  \right] \right]^{+} \right)^2 \\ 
	& \leq  \frac{4R^2}{\eta^2} +  T\left(  4R^2\delta  +  \frac{8}{\omega^2}(1+m)G^2 + 2C^2 + \frac{8GR}{\eta}  \right)   + T^2 \left(  2\delta \eta^2  \left( \frac{4}{\omega^2}(1+m)G^2 + C^2 \right)  + 8\delta \eta GR  \right) 
\end{align*} 
Note that the above bound also holds for a given $i \in [n]$ and $j \in \mathcal{N}_i$, that is, the R.H.S. of the above equation is also a bound for the term $\left( \left[\mathbb{E} \left[ \sum_{t=1}^{T} g_{ij}(\bx_i^{(t)} , \bx_j^{(t)})  \right] \right]^{+} \right)^2$.
Taking the square root of both sides, and using the fact that $\sqrt{\sum_{i=1}^{n} p_i} \leq \sum_{i=1}^{n} \sqrt{p_i}$ for positive $p_1,\hdots,p_n$, we get:
\begin{align} \label{eq:proof_10}
	 \mathbb{E} \left[ \sum_{t=1}^{T} g_{ij}(\bx_i^{(t)} , \bx_j^{(t)})  \right]  
	&  \leq  \frac{2R}{\eta} +  \sqrt{2T} \sqrt{\left(  2R^2\delta  +  \frac{4}{\omega^2}(1+m)G^2 + C^2 + \frac{4GR}{\eta}  \right)}  \notag   \\
	& \quad + \sqrt{2}T \sqrt{\left(  \delta \eta^2  \left( \frac{4}{\omega^2}(1+m)G^2 + C^2 \right)  + 4\delta \eta GR  \right)}
\end{align}
Dividing both sides of \eqref{eq:proof_10} by $T$, using the convexity of constraint function $g_{ij}$ and substituting $\eta = \frac{a}{\sqrt{T}}$, we get:
\begin{align*}
	\mathbb{E} \left[ g_{ij}(\blx_{i,avg}^{(T)} , \blx_{j,avg}^{(T)})  \right]  
	 & \leq   \frac{1}{T^{\frac{1}{4}}} \left( \sqrt{\frac{8GR}{a}} + \sqrt{8\delta aGR}  \right)  + \frac{2R}{a \sqrt{T}} +  \frac{1}{\sqrt{T}} \sqrt{2\left(  2R^2\delta + \frac{4}{\omega^2}(1+m)G^2 + C^2  \right)}   \\
	& \quad + \frac{1}{\sqrt{T}} \sqrt{2 \delta a^2  \left( \frac{4}{\omega^2} (1+m)G^2 + C^2 \right)}
\end{align*} 
This concludes the proof of \eqref{eq:thm_2} in Theorem \ref{thm:opt}

$\hfill \hspace{-0.6cm}\square$

\section{Decentralized compressed optimization with Bandit feedback} \label{sec:bandit_feedback}
\noindent In this section, we focus on the bandit feedback scenario where the nodes do not have direct access to samples drawn from their local data distributions. This could, as an example, arise in situations where the samples are high dimensional and thus can be hard to observe or measure. For the model we work with in this paper, we now assume that the nodes instead can query the value of the local objective function $f_i(\bx_i,\xi_i)$ for some particular choices of the parameter $\bx_i$. We first formally define the objective query process for the nodes and then describe how this model can be used to develop a stochastic gradient method for optimizing the overall objective \eqref{eq:main_obj}.

Let $\mathbb{S} := \{ \bu \in \mathbb{R}^d | \verts{\bu}_2 =1 \}$ and $\mathbb{B} := \{ \bu \in \mathbb{R}^d | \verts{\bu}_2 \leq 1 \} $ be the unit sphere and the unit ball in $d$-dimensions, respectively. For each node $i \in [n]$, and at any stage in the optimization process, we assume access to two local objective values $f_i(\bx_i \pm \zeta \bu_i, \xi_i )$ where $\bu_i$ is sampled uniformly at random over the unit sphere $ \mathbb{S}$ (independent of $\bx_i$ or $\xi_i$), $\zeta$ is a small positive constant, and $\bx_i$ is the local model parameter. To evaluate the gradient using these objective values, we make use of the following fact from \cite{FKM05}:
\begin{fact} \label{fact:bandit_interim1}
	Consider a function $\phi : \mathbb{R}^{d} \rightarrow \mathbb{R}$, and let $\zeta >0$. Define $\tilde{\phi}(\bx) = \bbE_{\bu \sim \mathcal{U}(\mathbb{B})} [ \phi(\bx + \zeta \bu) \bu ]$ where $\mathcal{U}(\mathbb{B})$ denotes uniform distribution over the unit ball $\mathbb{B} \subset \mathbb{R}^d$. The following hold:
	\begin{enumerate}[(i)]
		\item If $\phi$ is convex, then $\tilde{\phi}$ is also convex.
		\item For any $\bx \in \mathbb{R}^{d}$:
		\begin{align*}
			\nabla_{\bx} \tilde{\phi}(\bx) = \frac{d}{\zeta} \bbE_{\bu \sim \mathcal{U}(\mathbb{S})} [\phi(\bx+ \zeta \bu ) \bu ]
		\end{align*}
		where $\mathcal{U}(\mathbb{S})$ denotes the uniform distribution over the unit sphere $\mathbb{S} \subset \mathbb{R}^d$.	
\end{enumerate}
	 
\end{fact} 
For the node $i \in [n]$, the above fact can be used to estimate the gradient of the local objective function using the values $f_i(\bx_i \pm \zeta \bu_i, \xi_i)$ where $\bu_i \sim \mathcal{U}(\mathbb{S})$. For a given $\xi_i$, we define $\tilde{f}_i(\bx_i,\xi_i) = \bbE_{\bv_i \sim \mathcal{U}(\mathbb{B})} [ f_i(\bx_i + \zeta \bv_i, \xi_i) \bv_i ]  $. From the above fact, it also follows that $\tilde{f}(\bx_i,\xi_i)$ is convex in $\bx_i$ for a given $\xi_i$.

Note that as stated, the parameter vector $\bx_i \pm \zeta \bu_i$ may not lie in the feasible set $\mathcal{X}$ for all range of values of $\zeta$. Thus, we need some restriction on the range of values $\zeta$ can take. In the following, we make this argument precise. We first introduce an additional mild assumption on the topology of the set $\mathcal{X}$:
\begin{assumption} \label{assump:topology}
	The set $\mathcal{X}$ has a non-empty interior, that is, $\exists \by_0 \in \mathcal{X}$, $r >0$, s.t. $\mathcal{B}(\by_0,r) \subset \mathcal{X}$. Here, $\mathcal{B}(\by_0,r)$ denotes the open ball of radius $r$ centered at $\by_0$, i.e., $\mathcal{B}(\by_0,r) = \{ \bx | \verts{\bx - \by_0}_2 \leq r \}$
\end{assumption}  
From the above assumption, by the convexity of $\mathcal{X}$, it can also be concluded that for any $\alpha \in (0,1)$ and $\bx \in \mathcal{X}$, we have $\mathcal{B}((1-\alpha) \bx + \alpha \by_0 , \alpha r ) \subset \mathcal{X}$. We further define the set $\widetilde{\mathcal{X}} = \{ (1-\frac{\zeta}{r}) \bx+ \frac{\zeta}{r}\by_0 | \bx \in \mathcal{X}  \}$. It can now be readily checked that if $\bx_i \in \widetilde{\mathcal{X}}$ for the node $i$, then $\bx_i \pm \zeta \bu_i \in \mathcal{X}$, where $\bu_i$ is any point on the unit sphere $\mathbb{S}$.  Thus in the development of the algorithm below, we project the parameters onto the space $\widetilde{\mathcal{X}}$ to ensure that during the bandit feedback, the evaluated parameter $\bx_i \pm \zeta \bu_i$ for any node $i$ lies in the space $\mathcal{X}$. 
\subsection{Algorithm: Bandit Feedback}
We develop a saddle point algorithm for the bandit feedback scenario to find a saddle-point of the modified Lagrangian:
\begin{align} \label{eq:bandit_lagrangian}
	\tilde{	\mathcal{L}}(\bx,\blambda) = \sum_{i=1}^{n} \Big[\tilde{f}_i(\bx_i, \xi_i)  
	+ \sum_{j \in \mathcal{N}_i} \Big( \lambda_{ij} g_{ij}(\bx_i,\bx_j) - \frac{\delta \eta}{2} \lambda_{ij}^2\Big)  \Big]
\end{align}
The vector $\bx \in \widetilde{\mathcal{X}}^{n}$ represents the stacked node parameters and $\blambda$ represents the stacked dual variables. 
Here, the main difference from the modified Lagrangian in sample feedback case presented in \eqref{eq:sample_lagrangain} is that the objectives \{$f_i\}_{i=1}^{n}$ of the nodes are now replaced by the functions $\{\tilde{f}_i\}_{i=1}^{n}$. Importantly, the gradient of these functions can be computed via the result of Fact \ref{fact:bandit_interim1} which enables us to develop a primal-dual gradient algorithm to find the saddle point of \eqref{eq:bandit_lagrangian}. 

The gradient w.r.t. the primal variable $\bx$ is given by:
\begin{align} \label{eq:bandit_lagrangian_primal_interim1}
	\nabla_{\bx_i}  \tilde{	\mathcal{L}}(\bx,\blambda) &=  \sum_{j \in \mathcal{N}_i} \left[\lambda_{ij} \nabla_{\bx_i} g_{ij}(\bx_i,\bx_j) + \lambda_{ji} \nabla_{\bx_i} g_{ij}(\bx_j, \bx_i) \right] \notag \\
	& \quad  + \nabla_{\bx_i} \tilde{f}_i(\bx_i, \xi_i) 
\end{align}
Using the result from Fact \ref{fact:bandit_interim1}, for any $i\in [n]$ we have:
\begin{align*}
	\nabla_{\bx} \tilde{f}_i(\bx_i , \xi_i) = \frac{d}{2 \zeta} \bbE_{\bu_i \sim \mathcal{U}(\mathbb{S})} [ f(\bx_i + \zeta \bu_i , \xi_i)  - f(\bx_i - \zeta \bu_i , \xi_i)  ]\bu_i
\end{align*} 
As the node has access to the values of the local objective function in the bandit feedback scenario, the quantity $\frac{d}{2 \zeta}[f(\bx_i + \zeta \bu_i , \xi_i) - f(\bx_i - \zeta \bu_i , \xi_i)] $ for a given $\bu_i \sim \mathcal{U}(\mathbb{S})$, $\bx_i, \xi_i$, serves as an unbiased estimate of $\nabla_{\bx} \tilde{f}_i(\bx_i , \xi_i)$. Using this, we can construct the following estimate for the primal gradient $\nabla_{\bx_i}  \tilde{	\mathcal{L}}(\bx,\blambda)$:
\begin{align} \label{eq:bandit_p}
	\bp_i^{(t)}  := \frac{d}{2 \xi} \left[ f_i(\bx_i^{(t)} {+} \zeta\bu_i^{(t)} , \bxi_i^{(t)} ) {-} f_i(\bx_i^{(t)} {-} \zeta\bu_i^{(t)} , \bxi_i^{(t)} )  \right] \bu_i^{(t)}  + 2 \sum_{j \in \mathcal{N}_i} \lambda_{ij}^{(t)} \nabla_{\bx_i} g_{ij} (\bx_i^{(t)},\bx_j^{(t)})
\end{align}
The gradient of the Lagrangian in \eqref{eq:bandit_lagrangian} w.r.t. the dual parameter $\lambda_{ij}$ for $i \in [n]$ and $j \in \mathcal{N}_i$ is the same as in the sample feedback scenario and is given in \eqref{eq:sample_grad_primal}.

\begin{algorithm}[t!]
	\caption{Compressed Decentralized Optimization with Bandit Feedback}
	{\bf Initialize:} Random $\btx_i^{(1)} \in \widetilde{\mathcal{X}}$ individually for each $i \in [n]$ and $\lambda_{ij}^{(1)} = 0$ for each $j \in \mathcal{N}_i$. $\bhx_i^{(0)}=\mathbf{0}$ for each $i \in [n]$, number of iterations $T$, learning rate $\eta$, parameters $\zeta, \delta > 0$  \\
	(Communicate in the first iteration without compression to ensure that $\btx^{(1)} = \bhx^{(1)}$ )
	\vspace{0cm}
	\begin{algorithmic}[1] \label{algob}
		\FOR{$t=1$ \textbf{to} $T$ in parallel for $i \in [n]$}
		\STATE Compute $\bq_i^{(t)} = \C ( \btx_i^{(t)} - \bhx_i^{(t-1)} ) $	
		\FOR{nodes $k \in \mathcal{N}_i \cup \{i\}$} 	
		\STATE Send $\bq_i^{(t)}$ and receive $\bq_k^{(t)}$
		\STATE Update $\bhx_k^{(t)} = \bhx_k^{(t-1)} + \bq_k^{(t)} $ 
		\STATE Compute $\bx_k^{(t)} = \projXt (\bhx_k^{(t)}) $
		\ENDFOR
		\STATE Update running average for local parameter: \\
		$ \blx_{i,avg}^{(t)} = \frac{1}{t} \bx_i^{(t)} + \frac{t-1}{t} \blx_{i,avg}^{(t-1)} $
		\STATE Sample $\bu_i^{(t)} \sim \mathcal{U}(\mathbb{S})$
		\STATE Query the two values: $f_i(\bx_i^{(t)} \pm \zeta\bu_i^{(t)} , \bxi_i^{(t)} )$
		\STATE Compute the Lagrangian primal gradient estimate:
		\begin{align*}
			\bp_i^{(t)}  := 2 \sum_{j \in \mathcal{N}_i} \lambda_{ij}^{(t)} \nabla_{\bx_i} g_{ij} (\bx_i^{(t)},\bx_j^{(t)}) + \frac{d}{2 \xi} \left[ f_i(\bx_i^{(t)} {+} \zeta\bu_i^{(t)} , \bxi_i^{(t)} ) {-} f_i(\bx_i^{(t)} {-} \zeta\bu_i^{(t)} , \bxi_i^{(t)} )  \right] \bu_i^{(t)}  
		\end{align*}
		\STATE Update the primal variable via gradient descent:
		\begin{align*}
			\btx_i^{(t+1)}  = \projXt \left( \btx_i^{(t)}  - \eta \bp_i^{(t)}  \right)  
		\end{align*}
		\STATE For all $j \in \mathcal{N}_i$, update the dual variables via gradient ascent:
		\begin{align*}
			\lambda_{ij}^{(t+1)} = \left[ \lambda_{ij}^{(t)} + \eta \left( g_{ij}(\bx_i^{(t)} , \bx_j^{(t)} ) - \delta \eta \lambda_{ij}^{(t)}  \right)  \right]^{+}
		\end{align*}
		\ENDFOR
	\end{algorithmic}
{\bf Output:} Time averaged parameters $\blx_{i,avg}^{(T)}$ for all $i \in [n]$.
\end{algorithm}

The development of Algorithm~\ref{algob} is similar to that of Algorithm~\ref{algo}. The main difference is that we now find the saddle point of \eqref{eq:bandit_lagrangian} via alternating primal and dual variable gradient updates given in equations \eqref{eq:bandit_p} and \eqref{eq:sample_grad_dual} and project onto the space $\widetilde{\mathcal{X}}$ to ensure that the perturbed parameters lie in $\mathcal{X}$.
As before, for a node $i \in [n]$, $\btx_i$ refers to its raw parameter, $\bx_i$ as its local parameter, and $\bhx_i$ is the copy parameter.

We initialize the raw parameters $\{\btx_i^{(1)}\}_{i=1}^{n}$ inside the set $\widetilde{\mathcal{X}}$. During the first round, we assume the communication without compression to ensure that $\btx_i^{(1)} = \bhx_i^{(1)} $ for all $i \in [n]$.
 At time step $t\in [T]$, the node $i \in [n]$ computes and exchanges its copy parameters and constructs the local node parameter $\bx_i^{(t)}$ for which we track the running average (lines 2-8). As samples from the underlying distribution $\mathcal{P}_i$ are not directly revealed to the node in case of bandit feedback; instead it queries the value of the local objective $f_i (.,\xi_i)$ at parameters $\bx_i^{(t)}+\zeta \bu_i^{(t)}$ and $\bx_i^{(t)} - \zeta \bu_i^{(t)}$ where $\bu_i^{(t)}$ is uniformly sampled over the $d$-dimensional unit sphere $\mathbb{S}$ (lines 9-10). These values are then used to construct an unbiased estimate of $\nabla_{\bx_i}  \tilde{	\mathcal{L}}(\bx,\blambda)$ using \eqref{eq:bandit_p}, and then to update the raw parameter $\btx_i^{(t)}$ along with a projection operation back to the set $\widetilde{\mathcal{X}}$ (lines 11-13). Finally, the dual variables are also updated via gradient descent along with the projection to the positive real space to ensure feasibility (line 13). As in the case of sample feedback, the update of the dual steps in line 13 and the initialization $\lambda_{ij}^{(1)} = 0$ ensures that $\lambda_{ij}^{(t)} = \lambda_{ji}^{(t)}$ for all $t\in[T]$, and for all $i \in [n], j \in \mathcal{N}_i$.
\subsection{Main Result: Bandit Feedback}
\allowdisplaybreaks
{
We now present the convergence result rate for Algorithm~\ref{algob} which optimizes \eqref{eq:main_obj} in the bandit feedback scenario.
The proof details are provided in Section \ref{subsec:proof_bandit}.
\begin{theorem} \label{thm:opt_bandit}
	Consider running Algorithm \ref{algob} for $T$ iterations with fixed step size $\eta = \frac{a}{\sqrt{T}}$ for positive constant $a$, with perturbation constant $\zeta = \frac{1}{T}$, and regularization parameter $\delta$ taking values in $\left[ \frac{1 - \sqrt{1 - \frac{256 \eta^2 (1+m)\tilde{G}^2}{\omega^2} }}{4 \eta^2}, \frac{1 + \sqrt{1 - \frac{256 \eta^2 (1+m)\tilde{G}^2}{\omega^2}}}{4 \eta^2} \right]$, where $\omega \in (0,1)$ is the compression factor. Then, under Assumptions A.\ref{assump:X}-A.\ref{assump:topology}, for $T \geq \frac{256a^2 (1+m)\tilde{G}^2}{\omega^2}$, the expected value of $F$ evaluated at the stacked time averaged vector $\blx^{(T)}_{avg} := \frac{1}{T} \sum_{t=1}^{T}\bx^{(t)} $ satisfies:
	\begin{align} \label{eq:thmb_1}
		& \mathbb{E}[F(\blx^{(T)}_{avg})] -   F(\bx^{*}) \leq  \frac{2R^2}{ a \sqrt{T}}    + \frac{a}{\sqrt{T}} \left[ \frac{16}{\omega^2} d^2 (1+m)G^2 + C^2  \right]  +  \frac{2 \sqrt{m} \tilde{G}RC}{\delta a r \sqrt{T}} + \frac{4 R G}{rT} + \frac{4 \sqrt{n}G}{T}
	\end{align}
	where $r \leq \frac{R}{\sqrt{n}}$. For $i\in [n]$, $j \in \mathcal{N}_i$, the function $g_{ij}$ satisfies:
	\begin{align} \label{eq:thmb_2}
		 \mathbb{E} \left[ g_{ij}(\blx_{i,avg}^{(T)} , \blx_{j,avg}^{(T)})  \right] & \leq \frac{1}{T^{\nicefrac{1}{4}}} \left[\sqrt{\frac{8GR}{a}} + \sqrt{ \frac{8 \delta a ( R + r \sqrt{n}) G}{r} + 8GR\delta a    } \right] \notag \\
		& \quad  + \frac{1}{\sqrt{T}}    \sqrt{  \left( \frac{32}{\omega^2} d^2 (1{+}m)G^2 {+} 2C^2  \right)  {+}  \frac{4 \sqrt{m} \tilde{G} RC}{r \delta a^2}    {+} 4 R^2 \delta}  \notag \\
		&  \quad +  \frac{1}{\sqrt{T}}   \sqrt{   \delta a^2 \left( \frac{32}{\omega^2} d^2 (1+m)G^2 + 2C^2  \right)   +  \frac{4 \sqrt{m} \tilde{G} R C}{r}   }    \notag \\
		& \quad + \frac{1}{T^{\nicefrac{3}{4}}} \sqrt{  \frac{8(R+r\sqrt{n})G}{ra}  }
	\end{align}
	where $\blx_{k,avg}^{(T)}$ is time averaged parameter of node $k$ in $\blx_{avg}^{(T)}$.
\end{theorem}

The above result establishes that for a given compression requirement $\omega \in (0,1)$, the sub-optimality of the objective $\mathbb{E}[F(\blx^{(T)}_{avg})] -   F(\bx^{*})$ is $\mathcal{O} \left( \frac{1}{T^{\nicefrac{1}{2}}} \right) $. Similarly, the expected constraint violation for $i \in [n]$ and $j \in \mathcal{N}_i$ given by $\mathbb{E} \left[ g_{ij}(\blx_{i,avg}^{(T)} , \blx_{j,avg}^{(T)})  \right]$ is $\mathcal{O} \left( \frac{1}{T^{\nicefrac{1}{4}}} \right) $. Thus, in effect by choosing a large enough value of $T$, the number of iterations Algorithm \ref{algob} is run for, the obtained stacked parameter $\blx^{(T)}_{avg}$ is a good estimate of the optimal solution of the overall objective \eqref{eq:main_obj}. Moreover, the result obtained matches the rate that was obtained for the sample feedback case in Theorem~\ref{thm:opt}, where the nodes had access to the samples at every stage. Theorem \ref{thm:opt_bandit} thus establishes that even when node access to samples is not assumed, but rather only to a pair of values of the local objectives, the derived convergence rate suffers no degradation. 
}

\subsection{Convergence analysis} \label{subsec:proof_bandit}
In this section, we present the major proof details for establishing the convergence rates presented in Theorem \ref{thm:opt_bandit}.

\subsubsection{Preliminaries}
As done earlier for proof of bandit feedback, we use a compact notation by stacking together the parameters across the nodes. The modified Lagrangian in \eqref{eq:bandit_lagrangian} for a time step $t \in [T]$ in this notation is given as:

\begin{align} \label{eq:bandit_vector_lagrangian}
	\tilde{	\mathcal{L}}(\bx^{(t)},\blambda^{(t)}) = \tilde{f}(\bx^{(t)},\bxi^{(t)})+ \langle \blambda^{(t)} , \bg(\bx^{(t)}) \rangle - \frac{\delta \eta}{2}\Vert \blambda^{(t)} \Vert^2
\end{align}
where $\bx^{(t)}$, is of size $nd$, $\blambda^{(t)}$ is of size $m$, and $\bxi^{(t)}$ is collection of samples across all the nodes at time $t$.
For the proof we present below, we construct another quantity of interest:

\begin{align} \label{eq:bandit_vector_H}
	\calH(\bx^{(t)},\blambda^{(t)}) 
	= \widetilde{\mathcal{L}}(\bx^{(t)},\blambda^{(t)}) + \langle  \bp^{(t)} - \widetilde{\mathcal{L}}(\bx^{(t)}, \blambda^{(t)}) , \bx^{(t)}\rangle 
\end{align}
It can be seen that $\calH(\bx^{(t)}),\blambda^{(t)}$ is convex in the parameter $\bx^{(t)}$ and concave in $\blambda^{(t)}$ for any $t$. Further, the gradients of the function $\calH(\bx^{(t)}),\blambda^{(t)}$ satisfy:
\begin{align*}
	\nabla_{\bx} \calH(\bx^{(t)},\blambda^{(t)}) = \bp^{(t)}, \hspace{0.4cm}
	\nabla_{\blambda} \calH(\bx^{(t)},\blambda^{(t)}) = \nabla_{\blambda} \widetilde{\mathcal{L}}(\bx^{(t)},\blambda^{(t)})
\end{align*}

To derive our results, we consider another auxiliary result along the ones stated earlier in Section \ref{subsubsec:sample_prelim}.
\begin{fact} \label{prop:bound_f_diff}
	Under assumptions A.\ref{assump:convexity} and A.\ref{assump:gradient}, for all $t \in [T]$, $i \in [n]$ and any $\bu, \bv \in \mathcal{X}$, we have:
	\begin{align*}
		\bbE _{\bxi_i^{(t)}}[ f_i(\bu , \bxi_i^{(t)}) - f_i(\bv , \bxi_i^{(t)}) ]^2 \leq 4 G_i^2 \verts{\bu - \bv}^2 
	\end{align*}
	where $\bbE_{\bxi_i^{(t)}} [.]$ denotes expectation w.r.t. sampling at time-step $t$ for the node $i$. See Appendix \ref{sec:app_bandit} for proof.
\end{fact}

\subsubsection{Proof Of Theorem \ref{thm:opt_bandit}}
Similar to the proof of Theorem~\ref{thm:opt} presented in Section~\ref{subsec:proof_sample}, we start with the following lemma which establishes a relationship between the primal, dual variables in Algorithm~\ref{algob} and the function $\calH$ defined in \eqref{eq:bandit_vector_H}. This lemma can be seen as a counterpart of Lemma~\ref{lem:proof_sample_init} in the bandit feedback case.

\begin{lemma} \label{lem:proof_bandit_init}
	Consider the update steps in Algorithm~\ref{algob} with learning rate $\eta$. Under assumptions~A.\ref{assump:X}-A.\ref{assump:g}, for any $\bx \in \widetilde{\mathcal{X}}^{n}$ and $\blambda \in \mathbb{R}^m$ with $\blambda \succeq \mathbf{0}$, the summation of the function $\calH$ (defined in \eqref{eq:bandit_vector_H}) satisfies:
	\begin{align*}
		\sum_{t=1}^{T} \mathbb{E}	\left[\calH(\bx^{(t)},\blambda) - \calH( \bx,\blambda^{(t)})\right]  & \leq 
		\frac{\eta}{2} \sum_{t=1}^{T} \mathbb{E} \left( 2\Vert \bp^{(t)} \Vert^2  + \Vert \nabla_{\blambda}\widetilde{\mathcal{L}}_t(\bx^{(t)},\blambda^{(t)}) \Vert^2  \right)  \notag \\
		& \quad  + \frac{1}{2\eta} \sum_{t=1}^{T} \verts{\bx^{(t)} {-} \btx^{(t)} }^2  + \frac{1}{2 \eta} \sum_{t=1}^{T} \left(  \verts{\blambda }^2 + 4R^2 \right) 
	\end{align*}	
\end{lemma}

Now consider $\bx^{*} \in \mathcal{X}^n$, then by definition of $\widetilde{\mathcal{X}}$, we have $(1-\alpha)\bx^{*}+\alpha \tilde{\by}_0 \in \widetilde{\mathcal{X}}^n$ for $\alpha =  \frac{\zeta}{r} $ where $\tilde{\by}_0$ and $r$ are defined in assumption~A.\ref{assump:topology}\footnote{Here, $\tilde{\by}_0 \in \mathbb{R}^{nd}$ denotes the stacking of the $d$ dimensional vector $\by_0$ defined in assumption~A.\ref{assump:topology}}. Substituting $\bx =(1-\alpha)\bx^{*}+\alpha \tilde{\by}_0 $ in the result from Lemma~\ref{lem:proof_bandit_init} gives us:
\begin{align} \label{eq:proofb_14}
	\sum_{t=1}^{T} \mathbb{E}	\left[\calH(\bx^{(t)},\blambda) - \calH( (1-\alpha)\bx^{*}+\alpha \tilde{\by}_0,\blambda^{(t)})\right] & \leq 
	\frac{\eta}{2} \sum_{t=1}^{T} \mathbb{E} \left( 2\Vert \bp^{(t)} \Vert^2  + \Vert \nabla_{\blambda}\widetilde{\mathcal{L}}_t(\bx^{(t)},\blambda^{(t)}) \Vert^2  \right)  \notag \\
	& \quad + \frac{1}{2\eta} \sum_{t=1}^{T} \verts{\bx^{(t)} {-} \btx^{(t)} }^2  + \frac{1}{2 \eta} \sum_{t=1}^{T} \left(  \verts{\blambda }^2 + 4R^2 \right) 
\end{align}

The following result bounds the error $\bbE\verts{\be^{(t)}}^2 := \bbE\verts{\bx^{(t)} - \btx^{(t)}}^2$ for any time $t$ in terms of the summation of $\bbE \verts{\bp^{(t)}}$; see Appendix \ref{sec:app_bandit} for proof. 
\begin{lemma} \label{lem:e_tb} 
	Consider the error $\be^{(t)} := \bx^{(t)} - \btx^{(t)}$ for any $t\in[T]$. We have:
	\begin{align*}
		\mathbb{E}\Vert \be^{(t)}\Vert^2 \leq \frac{2 \eta^2}{\omega} \sum_{k=0}^{t-2} \left(  1 - \frac{\omega}{2} \right)^k \mathbb{E}\verts{\bp^{(t-k-1)}}^2
	\end{align*}
\end{lemma}
Using the bound from Lemma \ref{lem:e_tb} in \eqref{eq:proofb_14} yields:
\begin{align*}
	\sum_{t=1}^{T} \mathbb{E}	\left[\calH(\bx^{(t)},\blambda) - \calH((1-\alpha)\bx^{*}+\alpha \tilde{\by}_0,\blambda^{(t)})\right]  & \leq 
	\frac{\eta}{2} \sum_{t=1}^{T} \mathbb{E} \left( 2\Vert \bp^{(t)} \Vert^2  + \Vert \nabla_{\blambda}\widetilde{\mathcal{L}}_t(\bx^{(t)},\blambda^{(t)}) \Vert^2  \right)  \notag \\
	& + \frac{1}{2 \eta}  \left( 4R^2 +  \verts{\blambda}^2 \right)  {+} \frac{1}{2\eta} \sum_{t=1}^{T} \frac{2 \eta^2}{\omega} \sum_{k=0}^{t-2} \left(  1 {-} \frac{\omega}{2} \right)^k \mathbb{E}\verts{\bp^{(t-k-1)}}^2 
\end{align*}
Using the property of the double sum similar to the updates from \eqref{eq:temp_double_sum1} to \eqref{eq:temp_double_sum2} in the last term of the above equation, 
\begin{align} \label{eq:proofb_9}
	&\sum_{t=1}^{T} \mathbb{E}	\left[\calH(\bx^{(t)},\blambda) - \calH((1-\alpha)\bx^{*}+\alpha \tilde{\by}_0,\blambda^{(t)})\right]  \notag \\
	& \leq 
	\frac{\eta}{2} \sum_{t=1}^{T} \left( \left(2 {+} \frac{4}{\omega^2}\right)\Vert \bp^{(t)} \Vert^2  + \Vert \nabla_{\blambda}\widetilde{\mathcal{L}}_t(\bx^{(t)},\blambda^{(t)}) \Vert^2  \right) + \frac{1}{2 \eta}  \mathbb{E} \left( 4R^2 +  \verts{\blambda}^2 \right)  
\end{align}
We now provide bounds for the first and second terms on the R.H.S. of \eqref{eq:proofb_9} in Proposition \ref{prop:interim_1} below. The proof of this proposition is provided in Appendix \ref{sec:app_bandit}. 

\begin{prop} \label{prop:interim_1}
	For the update steps given in Algorithm \ref{algob}, under Assumptions A.\ref{assump:convexity}-A.\ref{assump:g}, for any $t \in [T]$, we have:
	\begin{enumerate}[(i)]
		\item $\mathbb{E}\verts{\bp^{(t)}}^2 \leq 4d^2(1+m)G^2 + 4(1+m)\tilde{G}^2 \mathbb{E} \verts{\blambda^{(t)}}^2$
		\item $\mathbb{E} \verts{\nabla_{\blambda} \widetilde{\mathcal{L}}_t(\bx^{(t)},\blambda^{(t)}) }^2 \leq 2C^2 + 2 \delta^2 \eta^2 \mathbb{E} \verts{\blambda^{(t)}}^2$
	\end{enumerate}
\end{prop}

%

Substituting the bounds from Proposition \ref{prop:interim_1} in \eqref{eq:proofb_9} and using that fact $\frac{2}{\omega^2} >1$, we have:
\begin{align} \label{eq:proofb_10}
	\sum_{t=1}^{T} \mathbb{E}	\left[\calH(\bx^{(t)},\blambda) - \calH((1-\alpha)\bx^{*}+\alpha \tilde{\by}_0,\blambda^{(t)})\right]  & \leq 
	\frac{1}{2 \eta}  \left( 4R^2 + \verts{\blambda}^2 \right)  + \eta T \left[  \frac{16}{\omega^2}  d^2 (1+m)G^2 + C^2  \right] \notag \\
	& \qquad  + \eta \left[ \frac{16}{\omega^2} (1+m) \tilde{G}^2 + \delta^2 \eta^2  \right] \sum_{t=1}^{T} \mathbb{E} \verts{\blambda^{(t)}}^2  
\end{align} 

We now express the L.H.S. of \eqref{eq:proofb_10} in terms of the Lagrangian $\widetilde{\mathcal{L}}$. This relation is provided in Proposition \ref{prop:interim_2} below, which is proved in Appendix \ref{sec:app_bandit}.
\begin{prop} \label{prop:interim_2}	
	For any $\blambda \in \mathbb{R}^{m} \text{ with } \blambda \succeq \mathbf{0}$, the updates of Algorithm \ref{algob} satisfy:
	\begin{align*}
		\sum_{t=1}^{T} \mathbb{E}	\left[\calH(\bx^{(t)},\blambda) - \calH((1-\alpha)\bx^{*}+\alpha \tilde{\by}_0,\blambda^{(t)})\right]  = \sum_{t=1}^{T} \mathbb{E}	\left[\widetilde{\mathcal{L}}_t(\bx^{(t)},\blambda) - \widetilde{\mathcal{L}}_t((1-\alpha)\bx^{*}+\alpha \tilde{\by}_0,\blambda^{(t)})\right] 
	\end{align*}
	where $\bx^{*}$ is the optimal parameter value for the objective \eqref{eq:main_obj}, and $\mathcal{H}$, $\widetilde{\mathcal{L}}$ are defined in \eqref{eq:bandit_vector_H} and \eqref{eq:bandit_vector_lagrangian}, respectively.
\end{prop}
The relation in Proposition \ref{prop:interim_2} implies the following for \eqref{eq:proofb_10}:
\begin{align*}
	\sum_{t=1}^{T} \mathbb{E}	\left[\widetilde{\mathcal{L}}_t(\bx^{(t)},\blambda) - \widetilde{\mathcal{L}}_t((1-\alpha)\bx^{*}+\alpha \tilde{\by}_0,\blambda^{(t)})\right] 
	& \leq 
	\frac{1}{2 \eta}  \left( 4R^2 + \verts{\blambda}^2 \right)  + \eta T \left[ \frac{16}{\omega^2} d^2 (1+m)G^2 + C^2  \right] \notag \\
	& \quad + \eta \left[ \frac{16}{\omega^2} (1+m) \tilde{G}^2 + \delta^2 \eta^2  \right] \sum_{t=1}^{T} \mathbb{E} \verts{\blambda^{(t)}}^2  
\end{align*}

Using the definition of $\widetilde{\mathcal{L}}$ from \eqref{eq:bandit_vector_lagrangian} on the L.H.S. of the above, and rearranging the terms, we have:
\begin{align} \label{eq:bandit_interim1}
	&\sum_{t=1}^{T} \mathbb{E}	\left[\widetilde{f}(\bx^{(t)},\bxi^{(t)}) - \widetilde{f}((1-\alpha)\bx^{*}+\alpha \tilde{\by}_0,\bxi^{(t)})\right]  - \frac{\delta \eta T}{2}  \verts{\blambda}^2 + \left\langle \blambda, \mathbb{E} \left[ \sum_{t=1}^{T}  \bg(\bx^{(t)}) \right]   \right\rangle  \notag \\
	& - \mathbb{E} \left[ \sum_{t=1}^{T} \left\langle \blambda^{(t)},\bg((1{-}\alpha)\bx^{*}{+}\alpha \tilde{\by}_0)  \right\rangle \right] \notag \\
	& \leq 
	\frac{1}{2 \eta}  \left( 4R^2 + \verts{\blambda}^2 \right)  + \eta T \left[ \frac{16}{\omega^2} d^2 (1+m)G^2 + C^2  \right] + \eta \left[ \frac{16}{\omega^2} (1+m) \tilde{G}^2 + \delta^2 \eta^2 - \frac{\delta}{2}  \right] \sum_{t=1}^{T} \mathbb{E} \verts{\blambda^{(t)}}^2  
\end{align}

We now focus on the last term in \eqref{eq:bandit_interim1} which has a coefficient of $\left( \frac{16}{\omega^2}(1+m)\tilde{G}^2 + \delta^2\eta^2 - \frac{\delta}{2} \right) $. We will now choose the value of $\delta$ such that this coefficient is negative. This requires choosing a value of $\delta$ such that $2\eta^2 \omega^2 \delta^2  - \omega^2 \delta +  32(1+m)\tilde{G}^2   \leq 0$. Since this is quadratic in $\delta$, it can be easily checked that the following range of $\delta$ satisfies the desired inequality:
	\begin{align*}
		\delta \in \left[ \frac{1 - \sqrt{1 - \frac{256 \eta^2 (1+m)\tilde{G}^2}{\omega^2} }}{4 \eta^2}, \frac{1 + \sqrt{1 - \frac{256 \eta^2 (1+m)\tilde{G}^2}{\omega^2}}}{4 \eta^2} \right]
	\end{align*}
	To ensure that the values of the interval bounds are real, we require running the algorithm for $T \geq \frac{256a^2 (1+m)\tilde{G}^2}{\omega^2}$.
	Similar to what we had earlier in the the case of sample feedback, for $T \rightarrow \infty$ (i.e., $\eta \rightarrow 0$ for the choice $\eta = \frac{a}{\sqrt{T}}$), the left interval value of $\delta$ converges to a positive constant, meaning that we can still choose a finite value of $\delta$ in the infinitesimally small learning rate regime so that our bounds derived below are not vacuous. Choosing the value of $\delta$ in the prescribed range, and plugging it in \eqref{eq:bandit_interim1} yields:

\begin{align} \label{eq:proofb_11}
	&\sum_{t=1}^{T} \mathbb{E}	\left[\widetilde{f}(\bx^{(t)},\bxi^{(t)}) - \widetilde{f}((1-\alpha)\bx^{*}+\alpha \tilde{\by}_0,\bxi^{(t)})\right] - \frac{\delta \eta T}{2}  \verts{\blambda}^2 + \left\langle \blambda, \mathbb{E}  \sum_{t=1}^{T}  \bg(\bx^{(t)})   \right\rangle   - \mathbb{E} \left[ \sum_{t=1}^{T} \left\langle \blambda^{(t)},\bg((1{-}\alpha)\bx^{*}{+}\alpha \tilde{\by}_0)  \right\rangle \right] \notag \\
	&  \leq 
	\frac{1}{2 \eta}  \left( 4R^2 +  \verts{\blambda}^2 \right)  + \eta T \left[ \frac{16}{\omega^2} d^2 (1+m)G^2 + C^2  \right]   
\end{align} 
Our goal is to derive a bound for the sub-optimality of the function $F(\bx^{(t)})$. To this end, we will now provide bounds for the terms on the L.H.S. of \eqref{eq:proofb_11} in terms of the function $F$. We first consider the first term on the L.H.S. of \eqref{eq:proofb_11}. From the definitions of $f$ and $ \widetilde{f}$ provided in Fact \ref{fact:bandit_interim1}, we note:
\begin{align} \label{eq:interim_5}
	&\bbE \left[ \vert \widetilde{f}(\bx^{(t)} ,\bxi^{(t)} ) - f(\bx^{(t)}, \bxi^{(t)}) \vert \right] \notag \\
	& \stackrel{(a)}{=} \bbE \left[ \left| \sum_{i=1}^{n} f_i(\bx_i^{(t)} + \zeta \bu_i^{(t)} ,\bxi_i^{(t)} ) - f_i(\bx_i^{(t)}, \bxi_i^{(t)}) \right| \right] \notag \\
	& \stackrel{(b)}{\leq}   \bbE \left[ \sum_{i=1}^{n} \left|  f_i(\bx_i^{(t)} + \zeta \bu_i^{(t)} ,\bxi_i^{(t)} ) - f_i(\bx_i^{(t)}, \bxi_i^{(t)}) \right| \right] \notag \\
	& \stackrel{(c)}{\leq} \bbE \left[ \sum_{i=1}^{n} \sqrt{\bbE_{\bxi_i^{(t)}} \left[   f_i(\bx_i^{(t)} + \zeta \bu_i^{(t)} ,\bxi_i^{(t)} ) - f_i(\bx_i^{(t)}, \bxi_i^{(t)})  \right]^2} \right] 
\end{align}
where in (a), $\{\bu_i^{(t)}\}_{i=1}^{n}$ denote random vectors uniformly distributed over $\mathbb{B}$, (b) uses the triangle inequality, (c) uses the fact $\bbE [A] \leq \sqrt{\bbE[A^2]}$ via Jensen's inequality.

From Proposition \ref{prop:bound_f_diff} and the fact $\|\bu_i^{(t)}\|^2 = 1$ for all $i \in [n]$ (as they lie on the unit sphere $\mathbb{S}$), we have:
\begin{align} \label{eq:interim_4} 
	\bbE _{\bxi_i^{(t)}}[ f_i(\bx_i^{(t)} + \zeta \bu_i^{(t)} , \bxi_i^{(t)}) - f_i(\bx_i^{(t)} , \bxi_i^{(t)}) ]^2 \leq 4 G_i^2 \zeta^2  
\end{align}
Plugging the bound from \eqref{eq:interim_4} in \eqref{eq:interim_5} and noting that $\sum_{i=1}^{n} G_i \leq \sqrt{n} G$  (using the fact that $G^2 = \sum_{i=1}^{n} G_i^2 $) gives:
\begin{align*}
	\bbE \left[ \vert \widetilde{f}(\bx^{(t)} ,\bxi^{(t)} ) - f(\bx^{(t)}, \bxi^{(t)}) \vert \right] \leq 2 \zeta \sqrt{n}G
\end{align*}
Using Jensen's inequality for the L.H.S. of above equation and rearranging the terms finally yields:
\begin{align} \label{eq:interim_6}
	  \bbE [\widetilde{f}(\bx^{(t)} ,\bxi^{(t)} )]   \geq \bbE [F(\bx^{(t)})] - 2 \zeta \sqrt{n}G
\end{align}

The steps to bound the second term on the L.H.S. of \eqref{eq:proofb_11} are similar. We note that:
\begin{align} \label{eq:interim_8}
	&\bbE \left[ \vert \widetilde{f}((1-\alpha)\bx^{*}+\alpha \tilde{\by}_0,\bxi^{(t)}) - f(\bx^{*}, \bxi^{(t)}) \vert \right] \notag \\
	& \stackrel{(a)}{=} \bbE \left[ \left| \sum_{i=1}^{n} f_i((1-\alpha)\bx_i^{*}+\alpha \by_0 + \zeta \bu_i^{(t)},\bxi^{(t)}) - f_i(\bx_i^{*}, \bxi_i^{(t)}) \right| \right] \notag \\
	& \stackrel{(b)}{\leq}   \bbE \left[ \sum_{i=1}^{n} \left|  f_i((1-\alpha)\bx_i^{*}+\alpha \by_0 + \zeta \bu_i^{(t)},\bxi^{(t)}) - f_i(\bx_i^{*}, \bxi_i^{(t)}) \right| \right] \notag \\
	& = \bbE  \sum_{i=1}^{n} \bbE_{\bxi_i^{(t)}}  |  f_i((1{-}\alpha)\bx_i^{*} {+} \alpha \by_0 {+} \zeta \bu_i^{(t)},\bxi^{(t)}) - f_i(\bx_i^{*}, \bxi_i^{(t)}) |  \notag \\
	& \stackrel{(c)}{\leq} \bbE  \sum_{i=1}^{n} \left( \bbE_{\bxi_i^{(t)}} \left[  f_i((1-\alpha)\bx_i^{*} + \alpha \by_0 + \zeta \bu_i^{(t)},\bxi^{(t)}) - f_i(\bx_i^{*}, \bxi_i^{(t)}) \right]^2 \right)^{\nicefrac{1}{2}}  
\end{align}
where (a), (b) and (c) are the same arguments as what we used in arriving at \eqref{eq:interim_5}. Further using Proposition \ref{prop:bound_f_diff}, we have:
\begin{align} \label{eq:interim_7}
	& \bbE_{\bxi_i^{(t)}} \left[   f_i((1-\alpha)\bx_i^{*}+\alpha \by_0 + \zeta \bu_i^{(t)},\bxi^{(t)}) - f_i(\bx_i^{*}, \bxi_i^{(t)})  \right]^2  \leq 4 G_i^2 \verts{ -\alpha \bx_i^{*}+\alpha \by_0 + \zeta \bu_i^{(t)}  }^2 
\end{align}
Plugging in the bound from \eqref{eq:interim_7} into \eqref{eq:interim_8} and using Fact \ref{fact:proj_bound} for $\bx_i^{*}, \by_0 \in \mathcal{X}$ and $\verts{\bu_i^{(t)}} = 1$ for all $i \in [n]$, we have:
\begin{align*}
	& \bbE \left[ \vert \widetilde{f}((1-\alpha)\bx^{*}+\alpha \tilde{\by}_0,\bxi^{(t)}) - f(\bx^{*}, \bxi^{(t)}) \vert \right]  \leq 4G \alpha R + 2 \zeta G \sqrt{n}
\end{align*}
Using Jensen's inequality for the L.H.S. and rearranging:
\begin{align} \label{eq:interim_9}
	\bbE [\widetilde{f}((1-\alpha)\bx^{*}+\alpha \tilde{\by}_0,\bxi^{(t)})] \leq F(\bx^{*}) + 4G \alpha R + 2 \zeta G \sqrt{n}
\end{align}

Further, we can also simplify other terms on the L.H.S. of \eqref{eq:proofb_11}. We note that:
\begin{align} \label{eq:interim_10}
	 \sum_{t=1}^{T} \left\langle \blambda^{(t)},\bg((1{-}\alpha)\bx^{*}{+}\alpha \tilde{\by}_0)  \right\rangle  & = \sum_{t=1}^{T} \left\langle \blambda^{(t)},\bg(\bx^{*})  \right\rangle + \sum_{t=1}^{T} \left\langle \blambda^{(t)},\bg((1{-}\alpha)\bx^{*}{+}\alpha \tilde{\by}_0) - \bg(\bx^{*})  \right\rangle \notag \\
	 & \leq \sum_{t=1}^{T}  \verts{\blambda^{(t)}} \verts{\bg((1{-}\alpha)\bx^{*}{+}\alpha \tilde{\by}_0) - \bg(\bx^{*})  }
\end{align}
where to obtain the last inequality we have used the fact that $\langle \blambda^{(t)}, \bg(\bx^{*}) \rangle \leq 0$ for all $t \in [T]$ and the Cauchy-Schwarz inequality. For the second term in the product on the R.H.S. in \eqref{eq:interim_10}, using \eqref{assump:gradient_constraint} in Assumption \ref{assump:gradient} $g(\bx_i, \bx_j)$ are $G_{ij}$-Lipschitz for all $i,j \in [n]$, we have:
\begin{align} \label{eq:interim_11}
	 \verts{\bg((1{-}\alpha)\bx^{*}{+}\alpha \tilde{\by}_0) - \bg(\bx^{*})  }^2 & \leq \sum_{i=1}^{n} \sum_{j \in \mathcal{N}_i}^{n} G^2_{ij} \verts{{-}\alpha \bx^{*} {+} \alpha \tilde{\by}_0 }^2 
	  \leq 4 \alpha^2 R^2m\tilde{G}^2 
\end{align}
where $\tilde{G} := \max_{i \in [n], j \in \mathcal{N}_i } G_{ij}$ and the last inequality follows from noting that $\bx^{*}, \tilde{\by}_0 \in \mathcal{X}^n$ and using Fact \ref{fact:proj_bound}.
We now bound the first term in the product on the R.H.S. in \eqref{eq:interim_10}. From the update equation of $\blambda^{(t)}$ in line 13 of Algorithm \ref{algob}, we have:
\begin{align*}
	\verts{ \blambda^{(t+1)}}  \leq \verts{ \blambda^{(t)} + \eta  \nabla_{\blambda} \widetilde{\mathcal{L}}_t (\bx^{(t)} , \blambda^{(t)}) }  \leq (1-\delta \eta^2) \verts{\blambda^{(t)}} + \eta C
\end{align*}
where the second inequality follows from the gradient update for the dual variable \eqref{eq:sample_grad_dual}, the triangle inequality, the fact that $\delta \eta^2 \leq 1$ (since an upper bound for the range of allowed $\delta$ is $\frac{2}{4 \eta^2}$) and Assumption \ref{assump:g} to bound $\verts{\bg(\bx^{(t)})}_2$. Continuing the recursion till $t=1$, it can be shown that $\verts{\blambda^{(t)}} \leq \frac{C}{\delta \eta}$, $\forall t \in [T]$. Using this bound, and \eqref{eq:interim_11} in \eqref{eq:interim_10} leads to:
\begin{align} \label{eq:interim_13}
	\sum_{t=1}^{T} \left\langle \blambda^{(t)},\bg((1{-}\alpha)\bx^{*}{+}\alpha \tilde{\by}_0)  \right\rangle \leq \frac{2 \alpha R C \sqrt{m} \tilde{G} T }{\delta \eta}
\end{align} 
Finally, using the bounds from \eqref{eq:interim_6}, \eqref{eq:interim_9}, \eqref{eq:interim_13} in \eqref{eq:proofb_11} yields:
\begin{align} \label{eq:proofb_12}
	&\sum_{t=1}^{T} \mathbb{E}	\left[ F(\bx^{(t)}) - F(\bx^{*}) \right] - \frac{\delta \eta T}{2}  \verts{\blambda}^2  + \left\langle \blambda, \mathbb{E}  \sum_{t=1}^{T}  \bg(\bx^{(t)})   \right\rangle   - \mathbb{E} \left[ \sum_{t=1}^{T} \left\langle \blambda^{(t)},\bg((1{-}\alpha)\bx^{*}{+}\alpha \tilde{\by}_0)  \right\rangle \right] \notag \\
	&  \leq 
	\frac{1}{2 \eta}  \left( 4R^2 +  \verts{\blambda}^2 \right)  + \eta T \left[ \frac{16}{\omega^2} d^2 (1+m)G^2 + C^2  \right]  + 4G \alpha RT + 4 \zeta G \sqrt{n} T + \frac{2 \alpha RC \sqrt{m}\tilde{G}T}{\delta \eta}
\end{align}

Setting $\blambda = \frac{\left[ \mathbb{E} \left[ \sum_{t=1}^{T} \bg(\bx^{(t)})  \right]  \right]^{+}}{\delta \eta T + \frac{1}{\eta}}$, and
plugging it in \eqref{eq:proofb_12} gives:
\begin{align} \label{eq:proofb_13}
	&\sum_{t=1}^{T} \left(  \mathbb{E} \left[ F(\bx^{(t)}) \right] - F(\bx^{*})  \right) + \sum_{i=1}^{n} \sum_{j \in \mathcal{N}_i} \frac{\left( \left[\mathbb{E} \left[ \sum_{t=1}^{T} g_{ij}(\bx_i^{(t)} , \bx_j^{(t)})  \right] \right]^{+} \right)^2}{2\left( \delta \eta  T + \frac{1}{ \eta} \right)}  \notag \\
	&\leq \frac{2R^2}{ \eta}    + \eta T \left[ \frac{16}{\omega^2} d^2 (1+m)G^2 + C^2  \right] +  2 \sqrt{m} \tilde{G}\alpha R \frac{CT}{\delta \eta}  + 4 \alpha R GT + 4 \zeta \sqrt{n}GT
\end{align}
Dividing both sides of \eqref{eq:proofb_13} by $T$ and noting that the second term on the L.H.S. of \eqref{eq:proofb_13} is positive, we can bound the time-average sub-optimality of $F$ as:
\begin{align*}
	& \sum_{t=1}^{T} \frac{\left(  \mathbb{E} \left[ F(\bx^{(t)}) \right] - F(\bx^{*})  \right)}{T} \leq  \frac{2R^2}{ \eta T}    + \eta \left[ \frac{16}{\omega^2} d^2 (1+m)G^2  \right] + C^2 \eta  +  2 \sqrt{m} \tilde{G}\alpha R \frac{C}{\delta \eta} + 4 \alpha R G + 4 \zeta \sqrt{n}G
\end{align*}

Using the convexity of $F$ and setting $\eta = \frac{a}{\sqrt{T}}$, $\zeta = \frac{1}{T}$ and $\alpha = \frac{1}{r T}$ for some positive constant $a$, $r$, then for time averaged parameter $\blx^{(T)}_{avg} := \frac{1}{T} \sum_{t=1}^{T} \bx^{(t)} $ we have:
\begin{align} \label{eq:opt_resultb}
	&\mathbb{E}[F(\blx^{(T)}_{avg})] {-}   F(\bx^{*}) \leq  \frac{2R^2}{ a \sqrt{T}}    + \frac{a}{\sqrt{T}} \left[ \frac{16}{\omega^2} d^2 (1{+}m)G^2 {+} C^2  \right]   +  \frac{2 \sqrt{m} \tilde{G}RC}{\delta a r \sqrt{T}} + \frac{4 R G}{rT} + \frac{4 \sqrt{n}G}{T}
\end{align}
This concludes the proof for the suboptimality of the function $F$ given in \eqref{eq:thmb_1} in Theorem~\ref{thm:opt_bandit}. 

We now consider the expected constraint violations. From Fact \ref{fact:F_diff_bound}, we have that $\forall \bx \in \mathcal{X}^n,$  $\mathbb{E}[F(\bx)]-F(\bx^{*}) > -4GR$. Using this relation in \eqref{eq:proofb_13} gives:
\begin{align*}
	 &\sum_{i=1}^{n} \sum_{j \in \mathcal{N}_i} \left( \left[\mathbb{E} \left[ \sum_{t=1}^{T} g_{ij}(\bx_i^{(t)} , \bx_j^{(t)})  \right] \right]^{+} \right)^2  
	 \leq \\
	 &  T \left[  \left( \frac{32}{\omega^2} d^2 (1+m)G^2 + 2C^2  \right)  + 4 \sqrt{m} \tilde{G}\alpha R \frac{C}{\delta \eta^2} + \frac{8(\alpha R + \zeta \sqrt{n}) G}{\eta} + 4 R^2 \delta + \frac{8GR}{\eta}   \right] \\
	 & + T^2 \left[  \delta \eta^2 \left( \frac{32}{\omega^2} d^2 (1+m)G^2 + 2C^2  \right)   +  4 \sqrt{m} \tilde{G} \alpha R C  + 8 \delta \eta (\alpha R + \zeta \sqrt{n}) G  + 8GR\delta \eta \right] + \frac{4R^2}{\eta^2}
\end{align*}
Note that the above bound also holds for a given $i \in [n]$ and $j \in \mathcal{N}_i$, that is, the R.H.S. of the above equation is also a bound for the term $\left( \left[\mathbb{E} \left[ \sum_{t=1}^{T} g_{ij}(\bx_i^{(t)} , \bx_j^{(t)})  \right] \right]^{+} \right)^2$.
Taking the square root of both sides and using the fact that $\sqrt{\sum_{i=1}^{n} c_i} \leq \sum_{i=1}^{n} \sqrt{c_i}$ for positive $c_1,\hdots,c_n$, we get:
\begin{align*}
	&\mathbb{E} \left[ \sum_{t=1}^{T} g_{ij}(\bx_i^{(t)} , \bx_j^{(t)})  \right] \leq \frac{2R}{\eta}  \\
	&   + \sqrt{T}  \left[  \left( \frac{32}{\omega^2} d^2 (1+m)G^2 + 2C^2  \right)  + 4 \sqrt{m} \tilde{G}\alpha R \frac{C}{\delta \eta^2} + \frac{8(\alpha R + \zeta \sqrt{n}) G}{\eta} + 4 R^2 \delta + \frac{8GR}{\eta}    \right]^{\nicefrac{1}{2}} \\
	& \qquad + T \left[  \delta \eta^2 \left( \frac{32}{\omega^2} d^2 (1+m)G^2 + 2C^2  \right)   +  4 \sqrt{m} \tilde{G} \alpha R C   + 8 \delta \eta (\alpha R + \zeta \sqrt{n}) G + 8GR\delta \eta  \right]^{\nicefrac{1}{2}}
\end{align*}
Dividing both sides of the above by $T$, using the convexity of constraint function $g_{ij}$, and substituting $\eta = \frac{a}{\sqrt{T}}$, $\zeta = \frac{1}{T}$ and $\alpha = \frac{1}{rT}$, finally gives:
\begin{align*}
	&\mathbb{E} \left[ g_{ij}(\blx_{i,avg}^{(T)} , \blx_{j,avg}^{(T)})  \right] \\
	&  \leq \frac{1}{T^{\nicefrac{1}{4}}} \left[\sqrt{\frac{8GR}{a}} + \sqrt{ \frac{8 \delta a ( R + r \sqrt{n}) G}{r} + 8GR\delta a    } \right] + \frac{1}{\sqrt{T}}    \sqrt{ \left( \frac{32}{\omega^2} d^2 (1+m)G^2 + 2C^2  \right)  +  \frac{4 \sqrt{m} \tilde{G} RC}{r \delta a^2}    + 4 R^2 \delta }  \\
	& \quad  +  \frac{1}{\sqrt{T}}  \sqrt{   \delta a^2 \left( \frac{32}{\omega^2} d^2 (1+m)G^2 + 2C^2  \right)   +  \frac{4 \sqrt{m} \tilde{G} R C}{r}   }  + \frac{1}{T^{\nicefrac{3}{4}}} \sqrt{  \frac{8(R+r\sqrt{n})G}{ra}  }
\end{align*}

This concludes the proof of \eqref{eq:thmb_2} in Theorem \ref{thm:opt_bandit}.
$\hfill \hspace{-0.6cm}\square$

\section{Experiments} \label{sec:expts}
We provide simulations to demonstrate the effectiveness of our proposed scheme for communication efficient optimization.

\begin{figure*}[htp]
	\centering
	\begin{subfigure}{.22\textwidth}
		\centering
		\includegraphics[scale=0.27]{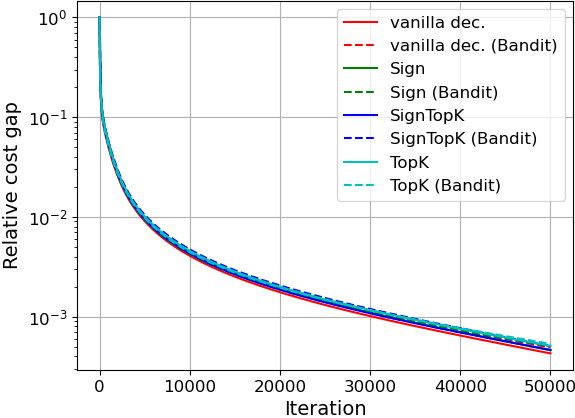}
		\caption{Comparison of relative cost gap $\frac{F(\bar{\mathbf{x}}^{(t)}) - F(\mathbf{x}^{*}) }{F(\bar{\mathbf{x}}^{(1)}) - F(\mathbf{x}^{*})}$ at iteration~$t$.}
		\label{fig1:rel_subopt}
	\end{subfigure} \hfill
	 \begin{subfigure}{.22\textwidth}
	 	\centering
	 	\includegraphics[scale=0.27]{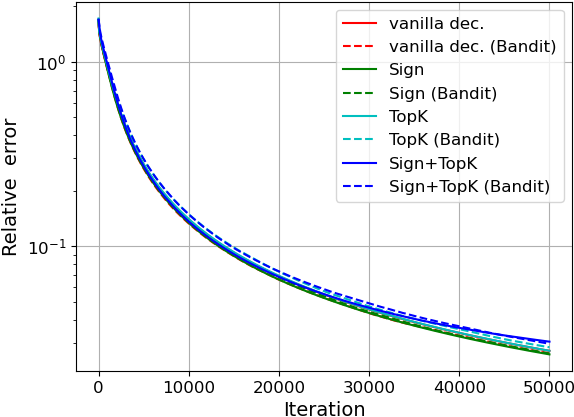}
	 	\caption{Comparison of relative parameter error $\frac{ \Vert \bar{\mathbf{x}}^{(t)} -\mathbf{x}^{*} \Vert }{ \Vert \mathbf{x}^{*} \Vert }$ at iteration $t$.}
	 	\label{fig1:rel_err}
	 \end{subfigure} \hfill
 \begin{subfigure}{.22\textwidth}
 	\centering
 	\includegraphics[scale=0.27]{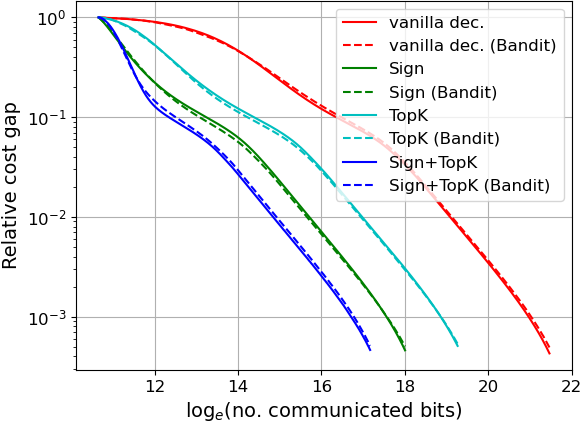}
 	\caption{Comparison of relative cost gap comparison for number of bits communicated for different schemes.}
 	\label{fig1:comm-subopt}
 \end{subfigure} \hfill
\begin{subfigure}{.22\textwidth}
	\centering
	\includegraphics[scale=0.27]{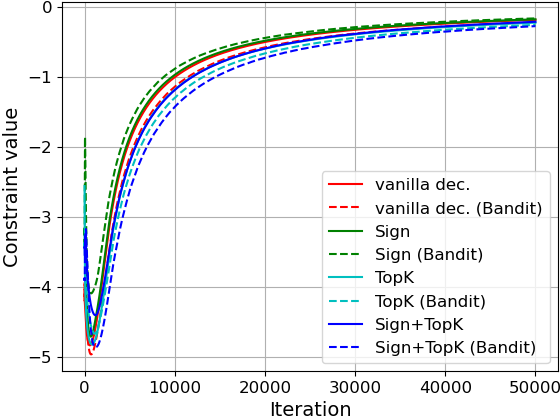}
	\caption{Constraint value $g_{ij}(\bar{\mathbf{x}}^{(t)}_i,\bar{\mathbf{x}}^{(t)}_j )$ for a randomly chosen edge $(i,j)$.}
	\label{fig1:vio}
\end{subfigure}
	\centering
	\caption{Performance comparison for various schemes on a decentralized QCQP objective in \eqref{eq:qcqp}}
	\label{fig1:main_qcqp}
\end{figure*}

\subsection{Setup and Hyperparameters}
We consider decentralized optimization on a randomly generated Erdos-Renyi graph of $n=30$ nodes with an edge probability of 0.15. For each node $i \in [n]$, we consider a quadratic objective given by $f_i(\bx) = \bx^{T} \mathbf{A}_i \bx + \mathbf{b}_i^{T}\bx $ where $\mathbf{A}_i \in \mathbb{R}^{10 \times 10}$ is sampled from a Wishart distribution with 10 degrees of freedom identity scaling matrix, and vector $\mathbf{b}_i$ is sampled from a Gaussian distribution with mean and variance drawn uniformly at random from the interval [0,1]. We consider the feasible parameter space $\mathcal{X}$ to be the Euclidean ball of radius $\frac{40}{\sqrt{30}}$ centered at the origin. For each $i \in [n], j \in \mathcal{N}_i$, we model the constraints on the node parameters as $g_{ij}(\bx_i,\bx_j) = \Vert \bx_i - \bx_j \Vert^2 +c_{ij} $ where $c_{ij}$ is independently drawn uniformly at random from $[-5,-3]$. The overall objective is thus given by:
\begin{align} \label{eq:qcqp}
	&\min_{ \{\bx_1,\dots,\bx_n\} \in \mathcal{X} } F(\bx) =  \sum_{i=1}^{n} \bx^{T}_i \mathbf{A}_i \bx_i + \mathbf{b}_i^{T}\bx_i \\
	&\text{s.t. } \Vert \bx_i - \bx_j \Vert^2_2 +c_{ij} \leq 0, \quad \forall i \in [n],j \in \mathcal{N}_i \notag 
\end{align}
where $\bx$ denotes concatenation of $\{ \bx_i,\hdots,\bx_n\}$. Note that choosing $c_{ij}\leq 0$ for all $i \in [n], j \in \mathcal{N}_i$ implies that the above QCQP has a non-empty feasible set. We take the learning rate as $\eta=0.001$, and choose $\delta = 100$, and run all considered schemes for a total of $5\times 10^{4}$ iterations. For gradient estimation in case of bandit feedback, we take $\zeta = 10^{-4}$. 

\subsection{Results}
The simulation results for optimizing objective \eqref{eq:qcqp} are presented in Figure~\ref{fig1:main_qcqp}, where we compare vanilla decentralized (no compression) algorithm with our proposed compressed optimization procedure using $Sign$ \cite{KQSJ19}, $TopK$ \cite{SCJ18} and composed $Sign+TopK$ \cite{BDKD19} compression operators. Schemes with \textit{`Bandit'} in parenthesis indicate those implemented via Algorithm~\ref{algob} for the case of gradient estimation in bandit feedback, and via Algorithm~\ref{algo} with sample feedback otherwise.  Figure \ref{fig1:rel_subopt} shows the relative cost gap for the objective given by
 $\frac{F(\bar{\mathbf{x}}^{(t)}) - F(\mathbf{x}^{*}) }{F(\bar{\mathbf{x}}^{(1)}) - F(\mathbf{x}^{*})}$, and Figure \ref{fig1:rel_err} shows the difference of the parameter from the optimal value normalized to the latter, given by $\frac{ \Vert \bar{\mathbf{x}}^{(t)} -\mathbf{x}^{*} \Vert }{ \Vert \mathbf{x}^{*} \Vert }$ for iteration $t$. From these figures, we see that schemes with compression, including the ones implemented via bandit feedback, effectively perform the same as uncompressed vanilla SGD to minimize the objective. The benefit of our proposed scheme can be seen in Figure \ref{fig1:comm-subopt}, where we show the relative cost gap as a function of the number of bits communicated among the nodes, assuming precision of 32bit floats. To achieve a target relative cost gap of around $10^{-3}$, compressed schemes use significantly fewer bits than vanilla decentralized training, saving a factor of about $7 \times$ with $TopK$ compression, factor of $30 \times$ when using $Sign$ compression operation, and a factor of around $50 \times$ for the composed $Sign+TopK$ compression operator. Further, in Figure \ref{fig1:vio} we plot the value of the constraint  $g_{ij}(\bar{\mathbf{x}}^{(t)}_i,\bar{\mathbf{x}}^{(t)}_j )  $ for a randomly chosen $i\in [n]$ and $j \in [n]$. The constraint value settles to a negative value, which implies that each scheme arrives at an objective value lying in the feasible space of the problem \eqref{eq:qcqp}. 
 
 In conclusion, our proposed schemes in Algorithms \ref{algo} and \ref{algob} for communication efficient decentralized optimization provide performance similar to that in the full precision vanilla decentralized method, while saving substantially in the total number of bits communicated among the nodes during the optimization process.

\section*{Acknowledgment}
{This research was supported in part by  the Army Research Laboratory under Cooperative Agreement Number W911NF-17-2-0196.
	 The views and conclusions contained in this document are those of the authors and should not be interpreted as representing the official policies, either expressed or implied, of the Army Research Laboratory or the U.S. Government. The U.S. Government is authorized to reproduce and distribute reprints for Government purposes notwithstanding any copyright notation here on.}

\bibliographystyle{alpha}

\begin{thebibliography}{199}
	
	\bibitem{NO09a}
	A. Nedic and A. Ozdaglar, ``Distributed subgradient methods for multiagent optimization," {\it IEEE Trans. Autom. Control}, vol. 54, no. 1, pp. 48-61, Jan. 2009.
	
	\bibitem{YLY16}
	K. Yuan, Q. Ling, and W. Yin, ``On the convergence of decentralized gradient descent,” {\it SIAM J. Optim.}, vol. 26, no. 3, pp. 1835-1854, 2016.
	
	\bibitem{DAW12}
	J. C. Duchi, A. Agarwal, and M. J. Wainwright, ``Dual averaging for distributed optimization: Convergence analysis and network scaling,” {\it IEEE Trans. Autom. Control}, vol. 57, no. 3, pp. 592-606, March 2012.
	
	\bibitem{SLYWY14}
	W. Shi, Q. Ling, K. Yuan, G. Wu, and W. Yin, ``On the linear convergence of the ADMM in decentralized consensus optimization,” {\it IEEE Trans. Signal Process.}, vol. 62, no. 7, pp. 1750-1761, April 2014.
	
	\bibitem{SLWY15}
	W. Shi, Q. Ling, G. Wu, and W. Yin, ``EXTRA: An exact first-order algorithm for decentralized consensus optimization,” {\it SIAM J. Optim.}, vol. 25, no. 2, pp. 944-966, 2015.
	
	\bibitem{SVKB17}
	M.O. Sayin, N.D. Vanli, S.S. Kozat, and T. Ba\c{s}ar, ``Stochastic subgradient algorithms for strongly convex optimization over distributed networks,” {\it IEEE Transactions on Network Science and Engineering}, 4(4):248-260, October-December 2017. 
	
	\bibitem{LLZ20}
	G. Lan, S. Lee, and Y. Zhou, ``Communication-efficient algorithms for decentralized and stochastic optimization,” {\it Math. Program.}, vol. 180, pp. 237-284, 2020.
	
	\bibitem{MC14}
	D. Mateos-Nunez and J. Cort\'{e}s, ``Distributed online convex optimization over jointly connected digraphs,” {\it IEEE Trans. Netw. Sci. Eng.}, vol. 1, no. 1, pp. 23-37, January–June 2014.
	
	\bibitem{AGL17}
	M. Akbari, B. Gharesifard, and T. Linder, ``Distributed online convex optimization on time-varying directed graphs,” {\it IEEE Trans. Control Netw. Syst.}, vol. 4, no. 3, pp. 417-428, September 2017.
	
	\bibitem{CRS14}
	J. Chen, C. Richard, and A. H. Sayed, ``Multitask diffusion adaptation over networks,” {\it IEEE Trans. Signal Process.}, vol. 62, no. 16, pp. 4129-4144, August 2014.
	
	\bibitem{NRFS16}
	R. Nassif, C. Richard, A. Ferrari, and A. H. Sayed, ``Multitask diffusion adaptation over asynchronous networks,” {\it IEEE Trans. Signal Process.}, vol. 64, no. 11, pp. 2835-2850, June 2016.
	
	\bibitem{KSR17}
	A. Koppel, B. M. Sadler, and A. Ribeiro, ``Proximity without consensus in online multiagent optimization,” {\it IEEE Trans. Signal Process.}, vol. 65, no. 12, pp. 3062-3077, June 2017.
	
	\bibitem{BKR19}
	A. S. Bedi, A. Koppel, and K. Rajawat, ``Asynchronous saddle point algorithm for stochastic optimization in heterogeneous networks,” {\it IEEE Trans. Signal Process.}, vol. 67, no. 7, pp. 1742-1757, April 2019.
	
	\bibitem{KBS07} 
	A. Kashyap, T. Ba\c{s}ar, and R. Srikant,  ``Quantized consensus,” {\it Automatica}, 43(7):1192-1203, July 2007.
	
	\bibitem{EB16} 
	S.R. Etesami and T. Ba\c{s}ar, ``Convergence time for unbiased quantized consensus over static and dynamic networks,” {\it IEEE Transactions on Automatic Control}, 61(2):443-455, February 2016.
	
	\bibitem{BEO16} 
	T. Ba\c{s}ar, S.R. Etesami, and A. Olshevsky, ``Convergence time of quantized Metropolis consensus over time-varying networks,” {\it IEEE Transactions on Automatic Control}, 61(12):4048-4054, December 2016.
	
	\bibitem{ELB16} 
	M. El Chamie, J. Liu, and T. Ba\c{s}ar, ``Design and analysis of distributed averaging with quantized communication,” {\it IEEE Transactions on Automatic Control}, 61(12):3870-3884, December 2016.
	
	\bibitem{ZC16}
	S. Zhu and B. Chen, ``Quantized consensus by the ADMM: Probabilistic versus deterministic quantizers,” {\it IEEE Trans. Signal Processing}, vol. 64, no. 7, pp. 1700-1713, April 2016
	
	\bibitem{ZYB19}
	J. Zhang, K. You, and T. Ba\c{s}ar, ``Distributed discrete-time optimization in multi-agent networks using only sign of relative state,” {\it IEEE Trans. Autom. Control}, vol. 64, no. 6, pp. 2352-2367, June 2019.
	
	\bibitem{CB21a} 
	X. Cao and T. Ba\c{s}ar, ``Decentralized online convex optimization based on signs of relative states,” {\it Automatica}, 129:109676, July 2021. 
	
	\bibitem{CB21b} 
	X. Cao and T. Ba\c{s}ar, ``Decentralized online convex optimization with event-triggered communications,” {\it IEEE Transactions on Signal Processing}, 69:284-299, 2021. 
	
	\bibitem{AHJKKR18}
	D. Alistarh, T. Hoefler, M. Johansson, S. Khirirat, N. Konstantinov, and C. Renggli, ``The convergence of sparsified gradient methods,” {\it Advances in Neural Information Processing Systems}, vol. 31, 2018.
	
	\bibitem{SCJ18}
	S. U. Stich, J.-B. Cordonnier, and M. Jaggi, ``Sparsified SGD with memory,” {\it Advances in Neural Information Processing Systems}, vol. 31, 2018, pp. 4452-4463.
	
	\bibitem{WWLZ18}
	J. Wangni, J. Wang, J. Liu, and T. Zhang, ``Gradient sparsification for communication-efficient distributed optimization,” {\it Advances in Neural Information Processing Systems}, vol. 31, 2018.
	
	
	\bibitem{RN05}
	M. G. Rabbat and R. D. Nowak, ``Quantized incremental algorithms for distributed optimization,” {\it IEEE J. Sel. Areas Commun.}, vol. 23, no. 4, pp. 798-808, April 2005.
	
	\bibitem{NOOT08}
	A. Nedic, A. Olshevsky, A. Ozdaglar, and J. N. Tsitsiklis, ``Distributed subgradient methods and quantization effects,” in {\it Proc. 47th IEEE Conf. Decision \& Control}, 2008, pp. 4177-4184.
	
	\bibitem{RMHP19}
	A. Reisizadeh, A. Mokhtari, H. Hassani, and R. Pedarsani, ``An exact quantized decentralized gradient descent algorithm,”  {\it IEEE Trans. Signal Process.}, vol. 67, no. 19, pp. 4934-4947, October 2019.
	
	\bibitem{AGLTV17}
	D. Alistarh, D. Grubic, J. Li, R. Tomioka, and M. Vojnovic, ``QSGD: Communication-efficient SGD via gradient quantization and encoding,” in {\it Proc. Adv. Neural Inf. Process. Syst.}, 2017, pp. 1709-1720.
	
	\bibitem{KSJ19}
	A. Koloskova, S. U. Stich, and M. Jaggi, ``Decentralized stochastic optimization and gossip algorithms with compressed communication,” in {\it Proc. Int. Conf. Mach. Learn.}, 2019, pp. 3478-3487.
	
	\bibitem{WHHZ18}
	J. Wu, W. Huang, J. Huang, and T. Zhang, “Error compensated quantized SGD and its applications to large-scale distributed optimization,” in International Conference on Machine Learning, pp. 5325-5333, 2018.
	
	\bibitem{BDKD19}
	D. Basu, D. Data, C. Karakus, and S. N. Diggavi, ``Qsparse-local-SGD: Distributed SGD with quantization, sparsification, and local computations,” {\it Advances in Neural Information Processing Systems}, vol. 32, 2019.
	
	\bibitem{CB20}
	X. Cao and T. Ba\c{s}ar, ``Decentralized multi-agent stochastic optimization with pairwise constraints and quantized communications,” {\it IEEE Trans. on Signal Processing}, vol. 68, pp. 3296-3311, 2020.
	
	\bibitem{AHU58}
	K. J. Arrow, L. Hurwicz, and H. Uzawa, {\it Studies in Linear and Non-linear Programming}, New York, NY, USA: Cambridge Univ. Press, 1958.
	
	\bibitem{NO09b}
	A. Nedi{\' c} and A. Ozdaglar, “Subgradient methods for saddle-point problems,” {\it J. Optim. Theory Appl.}, vol. 142, no. 1, pp. 205-228, 2009.
	
	\bibitem{CNS14}
	T.-H. Chang, A. Nedi{\' c}, and A. Scaglione, ``Distributed constrained optimization by consensus-based primal-dual perturbation method,”  {\it IEEE Trans. Autom. Control}, vol. 59, no. 6, pp. 1524-1538, June 2014.
	
	\bibitem{EZCLR19}
	M. Eisen, C. Zhang, L. F. Chamon, D. D. Lee, and A. Ribeiro, ``Learning optimal resource allocations in wireless systems,” {\it IEEE Trans. Signal Process.}, vol. 67, no. 10, pp. 2775-2790, May 2019.
	
	\bibitem{MJY12}
	M. Mahdavi, R. Jin, and T. Yang, ``Trading regret for efficiency: Online convex optimization with long term constraints,” {\it J. Mach. Learn. Res.}, vol. 13, pp. 2503-2528, 2012.
	
	\bibitem{CL19}
	X. Cao and K. J. R. Liu, ``Online convex optimization with time-varying constraints and bandit feedback,” {\it IEEE Trans. Autom. Control}, vol. 64, no. 7, pp. 2665-2680, July 2019.
	
	\bibitem{FKM05}
	A.D. Flaxman, A. T. Kalai, and H. B. McMahan, ``Online convex optimization in the bandit setting: Gradient descent without a gradient,” in {\it Proc. 16th Annu. ACM-SIAM Symp. Discrete Algorithms}, 2005, pp. 385-394.
	
	\bibitem{ADX10}
	A. Agarwal, O. Dekel, and L. Xiao, ``Optimal algorithms for online convex optimization with multi-point bandit feedback,” in {\it Proc. 23rd Annu. Conf. Learn. Theory}, 2010, pp. 28-40.
	
	\bibitem{DJWW15}
	J. C. Duchi, M. I. Jordan, M. J. Wainwright, and A. Wibisono, ``Optimal rates for zero-order convex optimization: The power of two function evaluations,” {\it IEEE Trans. Inf. Theory}, vol. 61, no. 5, pp. 2788-2806, May 2015.
	
	\bibitem{S17}
	O. Shamir, ``An optimal algorithm for bandit and zero-order convex optimization with two-point feedback,”  {\it  J. Mach. Learn. Res.}, vol. 18, no. 52, pp. 1-11, 2017.
	
	\bibitem{LKCTCA18}
	S. Liu, B. Kailkhura, P.-Y. Chen, P. Ting, S. Chang, and L. Amini, ``Zeroth-order stochastic variance reduction for nonconvex optimization,” in {\it Proc. Conf. Workshop Neural Inf. Process. Syst.}, 2018, pp. 3727-3737.
	
	\bibitem{HHG19}
	D. Hajinezhad, M. Hong, and A. Garcia, ``ZONE: Zeroth order nonconvex multi-agent optimization over networks,”  {\it IEEE Trans. Autom. Control}, vol. 64, no. 10, pp. 3995-4010, October 2019.
	
	\bibitem{KQSJ19}
	S. P. Karimireddy , Q. Rebjock, S. Stich, and M. Jaggi, ``Error feedback fixes signsgd and other gradient compression schemes," in {\it International Conference on Machine Learning}, pp. 3252-3261. PMLR, 2019.
	
	\bibitem{SDGD21}
	N. Singh, D. Data, J. George and S. Diggavi, ``SQuARM-SGD: Communication-Efficient Momentum SGD for Decentralized Optimization.", {\it IEEE Journal on Selected Areas in Information Theory}, 2021, 2(3), pp.954-969.
	
	\bibitem{MPPRRJ17}
	H. Mania, X. Pan, D. Papailiopoulos, B. Recht, K. Ramchandran and M. Jordan, ``Perturbed iterate analysis for asynchronous stochastic optimization", {\it SIAM Journal on Optimization}, vol. 27, np. 4, pp. 2202-2229, 2017.
	
	\bibitem{CLG17}
	T. Chen, Q. Ling and G. Giannakis, ``An online convex optimization approach to proactive network resource allocation", in {\it IEEE Transactions on Signal Processing}, vol. 65, no. 24, pp. 6350-6364, 2017.
	
	
	\bibitem{NN13}
	I. Necoara and V. Nedelcu, ``Rate analysis of inexact dual first-order methods application to dual decomposition", in {\it IEEE Transactions on Automatic Control}, vol. 59, no. 5, pp. 1232-1243, 2013.

	
	\bibitem{YN17}
	H. Yu and MJ. Neely, ``A Simple Parallel Algorithm with an O(1/t) Convergence Rate for General Convex Programs" in {\it SIAM Journal on Optimization}, vol. 27, no. 2, pp. 759-783, 2017.
	
	\bibitem{SDGD20}
	 N. Singh, D. data, J. George and S. Diggavi, ``SPARQ-SGD: Event-Triggered and Compressed Communication in Decentralized Optimization", in {\it IEEE Conference on Decision and Control} , pp. 3449-3456, 2020.
	
	
\end{thebibliography}

\clearpage
\newpage
\appendix
\section{Omitted Technical Details from Section \ref{sec:sample_feedback}} \label{sec:app_feedback}

\begin{lemma*} [Restating Lemma \ref{lem:proof_sample_init}]
	Consider the update steps in Algorithm~\ref{algo} with learning rate $\eta$ and parameter $\delta \geq 0$. Under assumptions~A.\ref{assump:X}-A.\ref{assump:g}, for any $\bx \in \mathcal{X}^{n}$ and $\blambda \in \mathbb{R}^m$ with $\blambda \succeq \mathbf{0}$, the summation of the Lagrangian function satisfies:
	\begin{align*} 
		\sum_{t=1}^{T} \mathbb{E}\left(\mathcal{L}(\bx^{(t)}, \blambda) - \mathcal{L}(\bx, \blambda^{(t)})\right) & \leq \frac{1}{2 \eta} \left(  \verts{\blambda }^2 + 4R^2 \right)  + \eta T \left( (1+m)G^2 + C^2 \right) + \frac{1}{2\eta} \sum_{t=1}^{T}    \bbE \verts{ \bx^{(t)} - \btx^{(t)} }^2  \notag \\
		& \quad + \eta \left( (1+m)\tilde{G}^2 + \delta^2\eta^2 \right)  \sum_{t=1}^{T} \mathbb{E}[ \Vert\blambda^{(t)}\Vert^2  ]
	\end{align*}
	where constants $G,C, \tilde{G}$ and $R$ are as defined in assumptions~A.\ref{assump:X}-A.\ref{assump:g}.
\end{lemma*}
\begin{proof}
	Using the gradient descent step of the raw variable $\btx$ from Algorithm~\ref{algo}, for any $\bx \in \mathcal{X}^n$, we have:
	\begin{align*}
		\Vert \bx - \btx^{(t+1)} \Vert^2 &= \verts{ \bx - \projXn \left( \btx^{(t)} - \eta \nabla_{\bx} \mathcal{L}(\bx^{(t)},\blambda^{(t)})  \right)  }^2 \\
		& \leq  \Vert \bx - \btx^{(t)} \Vert^2 + \eta^2 \Vert \nabla_{\bx} \mathcal{L}(\bx^{(t)},\blambda^{(t)}) \Vert^2   + 2 \eta \langle \bx - \btx^{(t)}, \nabla_{\bx} \mathcal{L}(\bx^{(t)},\blambda^{(t)}) \rangle 
	\end{align*}
	where the inequality is from Fact \ref{fact:proj_bound}. Rearranging terms:
	\begin{align} \label{eq:proof_1}
		\langle \btx^{(t)} {-} \bx , \nabla_{\bx} \mathcal{L}(\bx^{(t)},\blambda^{(t)}) \rangle \leq \frac{1}{2\eta} \left( \Vert \bx {-} \btx^{(t)} \Vert^2 {-} \Vert \bx {-} \btx^{(t+1)} \Vert^2 \right) + \frac{\eta}{2} \Vert \nabla_{\bx} \mathcal{L}(\bx^{(t)},\blambda^{(t)}) \Vert^2
	\end{align}
	Similarly, the gradient ascent updates for $\blambda \succeq \mathbf{0}$ gives:
	\begin{align*}
		\Vert \blambda - \blambda^{(t+1)} \Vert^2  &=   \verts{\blambda - \left[ \blambda^{(t)} + \eta \nabla_{\blambda} \mathcal{L}(\bx^{(t)},\blambda^{(t)}) \right]^{+}}^2 \\
		& \leq  \verts{\blambda - \blambda^{(t)}  }^2 + \eta^2 \verts{\nabla_{\blambda} \mathcal{L}(\bx^{(t)},\blambda^{(t)}}^2  - 2\eta \langle \blambda - \blambda^{(t)}, \nabla_{\blambda} \mathcal{L}(\bx^{(t)},\blambda^{(t)} \rangle 
	\end{align*}
	Rearranging terms we get:
	\begin{align} \label{eq:proof_2}
		\langle \blambda {-} \blambda^{(t)}, \nabla_{\blambda} \mathcal{L}(\bx^{(t)},\blambda^{(t)} \rangle & \leq \frac{1}{2\eta} \left( \verts{\blambda {-} \blambda^{(t)}  }^2 {-} \verts{\blambda {-} \blambda^{(t+1)}  }^2 \right)   + \frac{\eta}{2} \Vert \nabla_{\blambda} \mathcal{L}(\bx^{(t)},\blambda^{(t)}) \Vert^2
	\end{align}
	Since $\blambda \succeq \mathbf{0}$, $\mathcal{L}(\bx, \blambda^{(t)})$ is convex in $\bx$ and  $\mathcal{L}(\bx^{(t)}, \blambda)$ is concave in $\blambda$. We thus have:
	\begin{align} \label{eq:proof_3}
		\mathcal{L}(\bx^{(t)}, \blambda) & - \mathcal{L}(\bx, \blambda^{(t)})  \leq \langle \nabla_{\blambda} \mathcal{L}(\bx^{(t)},\blambda^{(t)})  , \blambda - \blambda^{(t)} \rangle  + \langle \nabla_{\bx} \mathcal{L}(\bx^{(t)},\blambda^{(t)})  , \bx^{(t)} - \bx \rangle \notag \\
		& = \langle \nabla_{\blambda} \mathcal{L}(\bx^{(t)},\blambda^{(t)})  , \blambda - \blambda^{(t)} \rangle + \langle \nabla_{\bx} \mathcal{L}(\bx^{(t)},\blambda^{(t)})  , \btx^{(t)} {-} \bx \rangle + \langle \nabla_{\bx} \mathcal{L}(\bx^{(t)},\blambda^{(t)})  , \bx^{(t)} - \btx^{(t)} \rangle \notag \\
		& \leq \langle \nabla_{\blambda} \mathcal{L}(\bx^{(t)},\blambda^{(t)})  , \blambda - \blambda^{(t)} \rangle + \langle \nabla_{\bx} \mathcal{L}(\bx^{(t)},\blambda^{(t)})  , \btx^{(t)} {-} \bx \rangle + \frac{\eta}{2} \verts{\nabla_{\bx} \mathcal{L}(\bx^{(t)},\blambda^{(t)})}^2  + \frac{1}{2\eta}   \verts{\bx^{(t)} - \btx^{(t)}}^2 
	\end{align}
	Substituting the bounds from \eqref{eq:proof_1}, \eqref{eq:proof_2} in \eqref{eq:proof_3}, summing from $t=1$ to $T$ and using $\blambda^{(1)} = \mathbf{0}$, we get:
	%
	%
	\begin{align*}
		\sum_{t=1}^{T} \left(\mathcal{L}(\bx^{(t)}, \blambda) - \mathcal{L}(\bx, \blambda^{(t)})\right)  & \leq \frac{1}{2 \eta} \left(  \verts{\blambda - \blambda^{(1)}  }^2 - \verts{\blambda - \blambda^{(T+1)}  }^2 \right) +  \frac{1}{2 \eta} \left( \verts{\bx {-} \btx^{(1)}}^2 {-} \verts{\bx {-} \btx^{(T+1)}}^2 \right) \\
		& \quad  + \frac{1}{2\eta} \sum_{t=1}^{T}    \verts{\bx^{(t)} {-} \btx^{(t)}}^2  + \frac{\eta}{2} \sum_{t=1}^{T} \left( \Vert \nabla_{\blambda} \mathcal{L}(\bx^{(t)},\blambda^{(t)}) \Vert^2 + 2 \Vert \nabla_{\bx} \mathcal{L}(\bx^{(t)},\blambda^{(t)}) \Vert^2  \right) 
	\end{align*}
	Taking the expectation, using bound for $\verts{\bx - \btx^{(1)}}_2^2$ from assumption~A.\ref{assump:X}  and using Fact \ref{fact:grad_lagrangian_bound} gives us:

	\begin{align} \label{eq:proof_6}
		\sum_{t=1}^{T} \mathbb{E}\left(\mathcal{L}(\bx^{(t)}, \blambda) - \mathcal{L}(\bx, \blambda^{(t)})\right)  & \leq \frac{1}{2 \eta} \left(  \verts{\blambda }^2 + 4R^2 \right)  + \eta T \left( (1+m)G^2 + C^2 \right) + \frac{1}{2\eta} \sum_{t=1}^{T}    \bbE \verts{ \btx^{(t)} - \bx^{(t)} }^2  \notag \\
		& \quad + \eta \left( (1+m)\tilde{G}^2 + \delta^2\eta^2 \right)  \sum_{t=1}^{T} \mathbb{E}[ \Vert\blambda^{(t)}\Vert^2  ]
	\end{align}
\end{proof}

\begin{fact*}[Restating Fact \ref{fact:grad_lagrangian_bound}]
	Consider the Lagrangian function over the primal and dual variables defined in \eqref{eq:lagrangian}. Then we have the following bounds:		
	\begin{enumerate} [(a)]
		\item $\bbE \Vert \nabla_{\blambda} \mathcal{L}_t(\bx^{(t)},\blambda^{(t)}) \Vert^2 \leq 2 C^2 + 2 \delta^2 \eta^2 \mathbb{E} \Vert \blambda^{(t)} \Vert^2$
		\item $\mathbb{E} \verts{\nabla_{\bx} \mathcal{L}_t(\bx^{(t)},\blambda^{(t)})}^2 \leq (1+m) \left( G^2  + 	\widetilde{G}^2 \mathbb{E} \verts{\blambda}^2  \right)$
	\end{enumerate}
where $C^2, \tilde{G} $  and $G $ are defined in Assumption \ref{assump:gradient}, \ref{assump:g}.
\end{fact*}

\begin{proof}
	From the definition of the Lagrangian, we note that:
	\begin{align} \label{eq:proof_13}
		\mathbb{E} \Vert \nabla_{\blambda} \mathcal{L}_t(\bx^{(t)},\blambda^{(t)}) \Vert^2 = \mathbb{E} \Vert  \bg(\bx^{(t)}) - \delta \eta \blambda^{(t)} \Vert^2 \leq 2 \mathbb{E} \Vert \bg(\bx^{(t)}) \Vert^2 + 2 \delta^2 \eta^2 \Vert \blambda^{(t)} \Vert^2 \leq 2 C^2 + 2 \delta^2 \eta^2 \mathbb{E} \Vert \blambda^{(t)} \Vert^2
	\end{align}
	where to get to \eqref{eq:proof_13}, we use $\verts{\bg(\bx)}^2  = \sum_{i=1}^{n} \sum_{j \in \mathcal{N}_i} (g_{ij} (\bx_i, \bx_j))^2 \leq  \sum_{i=1}^{n} \sum_{j \in \mathcal{N}_i} C_{ij}^2$ using Assumption \ref{assump:g}. This proves part (a). For proving part (b), note that from the definition of Lagrangian, we also have:
	\begin{align} \label{eq:proof_12}
		\mathbb{E} \verts{\nabla_{\bx} \mathcal{L}_t(\bx^{(t)},\blambda^{(t)})}^2 & = \mathbb{E} \verts{\nabla_{\bx} f(\bx^{(t)},\bxi^{(t)}) + \sum_{i=1}^{n} \sum_{j \in \mathcal{N}_i} \lambda_{ij}^{(t)} \nabla_{\bx} g_{ij}(\bx_i^{(t)},\bx_j^{(t)})}^2 \notag \\
		& \leq  \mathbb{E} \Big( \verts{\nabla_{\bx} f(\bx^{(t)},\bxi^{(t)})}  {+} \sum_{i=1}^{n} \sum_{j \in \mathcal{N}_i} \lambda_{ij}^{(t)} \verts{\nabla_{\bx} g_{ij}(\bx_i^{(t)},\bx_j^{(t)})}  \Big)^2 \notag \\
		& \leq (1+m) \mathbb{E} \verts{\nabla_{\bx} f(\bx^{(t)},\bxi^{(t)})}^2  + (1+m) \sum_{i=1}^{n} \sum_{j \in \mathcal{N}_i} (\lambda_{ij}^{(t)})^2 \mathbb{E}  \verts{\nabla_{\bx} g_{ij}(\bx_i^{(t)},\bx_j^{(t)})}^2  
	\end{align}
	The second term in R.H.S. of \eqref{eq:proof_12} can be bounded using the bound on $\Vert \nabla_{\bx} g_{ij}(\bx_i^{(t)},\bx_j^{(t)}) \Vert^2$  from \eqref{assump:gradient_constraint} in Assumption \ref{assump:gradient}. To bound the first term, by the law of total expectation we have:
	\begin{align*}
		\mathbb{E} \left[\verts{\nabla_{\bx} f(\bx^{(t)},\bxi^{(t)}) }^2\right] = \mathbb{E}_{\bx^{(t)}} \left[ \mathbb{E}_{\bxi^{(t)}} \left[ \verts{\nabla_{\bx} f(\bx^{(t)},\bxi^{(t)})}^2 | \bx^{(t)}   \right]   \right] \leq \mathbb{E}_{\bx^{(t)}} \left[ G^2 \right]
	\end{align*}
	where the inequality follows from \eqref{assump:gradient_func} in Assumption \ref{assump:gradient} and defining $G^2 = \sqrt{\sum_{i=1}^{n} G_i^2 }$. Substituting these bounds in \eqref{eq:proof_12} gives:
	\begin{align} \label{eq:proof_14}
		\mathbb{E} \verts{\nabla_{\bx} \mathcal{L}_t(\bx^{(t)},\blambda^{(t)})}^2 \leq (1+m) \left( G^2  + 	\tilde{G}^2 \mathbb{E} \verts{\blambda}^2  \right)
	\end{align} 
\end{proof}

\begin{fact} [Restating Fact \ref{fact:F_diff_bound}]
	For all $\bx \in \mathcal{X}^n,$ we have:
	\begin{align*}
		\mathbb{E}[F(\bx)]-F(\bx^{*}) > -4GR
	\end{align*}
where $R,G$ are defined in Assumptions \ref{assump:X} and \ref{assump:g}, respectively.
\end{fact}
\begin{proof}
	As $f(\bx,\bxi)$ is convex in $\bx$, $\forall \bx ,\btx \in \mathcal{X}^n$ and $\bxi \in \Xi$, we have:
		\begin{align*}
			f(\bx,\bxi) \geq f(\btx, \bxi) + \langle \nabla_{\bx} f(\btx,\bxi), \bx - \btx \rangle \\
			f(\btx,\bxi) \geq f(\bx, \bxi) + \langle \nabla_{\bx} f(\bx,\bxi), \btx - \bx \rangle
		\end{align*}
		Rearranging the terms in the above equations, we conclude:
		\begin{align*}
			\vert  f(\bx,\bxi) - f(\btx, \bxi) \vert &\leq \max \{ \langle \nabla_{\bx} f(\btx,\bxi), \btx - \bx \rangle  ,  \langle \nabla_{\bx} f(\bx,\bxi), \bx - \btx \rangle \} \\
			&\leq \left( \Vert  \nabla_{\bx} f(\bx,\bxi) \Vert + \Vert \nabla_{\bx} f(\btx,\bxi) \Vert  \right) \Vert \btx - \bx \Vert
		\end{align*}
		Taking expectation w.r.t. $\bxi$, noting that $F(\bx) = \mathbb{E}_{\bxi}[ f(\bx,\bxi)]$, and using Jensen's inequality, we get:
		\begin{align} \label{eq:proof_15}
			|F(\bx) - F(\btx)|  & \leq \Vert \btx - \bx \Vert \left( \sqrt{\mathbb{E} [ \Vert \nabla_{\bx} f(\bx,\bxi) \Vert^2 ] } + \sqrt{\mathbb{E} [ \Vert \nabla_{\bx} f(\btx,\bxi) \Vert^2 ] } \right) \notag \\
			& \leq 2G \Vert \btx - \bx \Vert
		\end{align}
		where for the last inequality we have used Assumption \ref{assump:gradient}. We finally note that $\Vert \btx - \bx \Vert \leq \Vert \btx \Vert + \Vert \bx \Vert \leq 2 R $ from Assumption \ref{assump:X} and $F(\bx^{*}) \leq F(\btx)$ for all $\btx \in \mathcal{X}^n$ (as $\bx^{*}$ is the minimizer of $F$) , which completes the proof.
\end{proof}

\begin{lemma*} [Restating Lemma \ref{lem:e_t}]
	For the update steps in Algorithm \ref{algo}, the norm of expected error $\mathbb{E}\Vert \be^{(t)} \Vert$ for any $t \in [T]$ is bounded as:
	\begin{align*}
		\mathbb{E}\Vert \be^{(t)} \Vert^2 \leq \frac{2 \eta^2}{\omega} \sum_{k=0}^{t-2} \left(  1 {-} \frac{\omega}{2} \right)^k \mathbb{E}\Vert\nabla_{\bx}\mathcal{L}_{t{-}1{-}k}(\bx^{(t{-}1{-}k)},\blambda^{(t{-}1{-}k)} ) \Vert^2
	\end{align*}
	where $\eta$ is the learning rate and $\omega$ is the compression coefficient.
\end{lemma*}

\begin{proof}
	
	From the definition of $\be^{(t)}$, we have:
	\begin{align} \label{eq:proof_lem_1}
		\verts{\be^{(t)}} =  \verts{\btx^{(t)} - \bx^{(t)} } =  \verts{\btx^{(t)} - \projXn \left(\bhx^{(t)}\right) } \leq  \verts{\btx^{(t)} -  \bhx^{(t)} }
	\end{align}
	where the last inequality follows from Fact \ref{fact:proj_bound}.
	Define $S^{(t)} := \mathbb{E} \verts{\btx^{(t)} -  \bhx^{(t)} }^2 $. From \eqref{eq:proof_lem_1}, we have that $\mathbb{E} \Vert \be^{(t)}\Vert^2 \leq S^{(t)}$. In the following, we will derive a bound for $S^{(t)}$.  Using the update step for $\bhx^{(t)}$, we can write:
	\begin{align*}
		S^{(t)}  & = \mathbb{E} \verts{\btx^{(t)} -  \bhx^{(t-1)}  - \mathcal{C}(\btx^{(t)} -  \bhx^{(t-1)}) }^2 \\
		& \stackrel{(a)}{\leq} (1-\omega) \mathbb{E} \verts{ \btx^{(t)} -  \bhx^{(t-1)} }^2 \\
		&= (1{-}\omega) \mathbb{E} \Vert   \projXn \left( \btx^{(t-1)} - \eta \nabla_{\bx}\mathcal{L}_{t-1}(\bx^{(t-1)},\blambda^{(t-1)} )  \right)  -  \bhx^{(t-1)} \Vert^2 \\
		&= (1{-}\omega) \mathbb{E} \Vert \btx^{(t-1)}   - \bhx^{(t-1)} - \btx^{(t-1)} + \projXn \left( \btx^{(t-1)} - \eta \nabla_{\bx}\mathcal{L}_{t-1}(\bx^{(t-1)},\blambda^{(t-1)} )  \right)   \Vert^2
	\end{align*}
	where (a) follows from the definition of compression operator (Definition \ref{def:comp}). For any $\alpha>0$, we have:\footnote{For any positive $\alpha$ and any vectors $\mathbf{a},\mathbf{b}$, we have that:
		\begin{align*}
			\verts{\mathbf{a}+\mathbf{b}}^2 \leq \left( 1+ \alpha \right) \verts{\mathbf{a}}^2 + \left(1+ \frac{1}{\alpha}\right) \verts{\mathbf{b}}^2
	\end{align*}}
	\begin{align*}
		 S^{(t)}   \leq (1-\omega)(1+\alpha)\mathbb{E}\verts{\btx^{(t-1)}   - \bhx^{(t-1)} }^2 + (1-\omega)\left(1+ \frac{1}{\alpha} \right)  \left[ \mathbb{E} \Big\Vert \projXn \left( \btx^{(t-1)} {-} \eta \nabla_{\bx}\mathcal{L}_{t-1}(\bx^{(t-1)},\blambda^{(t-1)} )  \right)   - \btx^{(t-1)}  \Big\Vert^2 \right]
	\end{align*}
	Noting that $\btx^{(t-1)} \in \mathcal{X}^n$ and using Fact \ref{fact:proj_bound} to bound the second term in R.H.S. of above and using $\alpha = \frac{\omega}{2}$, we get:
	\begin{align*}
		S^{(t)} & \leq \left( 1 - \frac{\omega}{2} \right)  \mathbb{E}\verts{\btx^{(t-1)}   - \bhx^{(t-1)} }^2 + \frac{2 \eta^2}{\omega} \mathbb{E}\verts{\nabla_{\bx}\mathcal{L}_{t-1}(\bx^{(t-1)},\blambda^{(t-1)} )}^2 \\
		& = \left( 1 - \frac{\omega}{2} \right)  S^{(t-1)} + \frac{2 \eta^2}{\omega} \mathbb{E}\verts{\nabla_{\bx}\mathcal{L}_{t-1}(\bx^{(t-1)},\blambda^{(t-1)} )}^2
	\end{align*} 
	Expanding the recurrent form for $S^{(t)}$ and noting that $S^{(1)}=0$, we get:
	\begin{align*}
		S^{(t)} \leq \frac{2 \eta^2}{\omega} \sum_{k=0}^{t-2} \left(  1 {-} \frac{\omega}{2} \right)^k \mathbb{E}\verts{\nabla_{\bx}\mathcal{L}_{t-1-k}(\bx^{(t-1-k)},\blambda^{(t-1-k)} )}^2
	\end{align*}
	Using the relation $\mathbb{E}\Vert \be^{(t)} \Vert^2_2 \leq S^{(t)}$, for all $t \in [T]$ we have:
	\begin{align*}
		\mathbb{E}\Vert \be^{(t)}\Vert^2 \leq \frac{2 \eta^2}{\omega} \sum_{k=0}^{t-2} \left(  1 - \frac{\omega}{2} \right)^k \mathbb{E}\verts{\nabla_{\bx}\mathcal{L}_{t-1-k}(\bx^{(t-1-k)},\blambda^{(t-1-k)} )}^2
	\end{align*}
	This concludes the proof of Lemma \ref{lem:e_t}.
	
\end{proof}

\section{Omitted Technical Details from Section \ref{sec:bandit_feedback}} \label{sec:app_bandit}

\begin{lemma*} [Restating Lemma~\ref{lem:proof_bandit_init}]
	Consider the update steps in Algorithm~\ref{algob} with learning rate $\eta$. Under assumptions~A.\ref{assump:X}-A.\ref{assump:g}, for any $\bx \in \widetilde{\mathcal{X}}^{n}$ and $\blambda \in \mathbb{R}^m$ with $\blambda \succeq \mathbf{0}$, the summation of the function $\calH$ (defined in \eqref{eq:bandit_vector_H}) satisfies:
	\begin{align*}
		\sum_{t=1}^{T} \mathbb{E}	\left[\calH(\bx^{(t)},\blambda) - \calH( \bx,\blambda^{(t)})\right] & \leq 
		\frac{\eta}{2} \sum_{t=1}^{T} \mathbb{E} \left( 2\Vert \bp^{(t)} \Vert^2  + \Vert \nabla_{\blambda}\widetilde{\mathcal{L}}_t(\bx^{(t)},\blambda^{(t)}) \Vert^2  \right)  \notag \\
		& \quad  + \frac{1}{2\eta} \sum_{t=1}^{T} \verts{\bx^{(t)} {-} \btx^{(t)} }^2  + \frac{1}{2 \eta} \sum_{t=1}^{T} \left(  \verts{\blambda }^2 + 4R^2 \right) 
	\end{align*}	
\end{lemma*}
\begin{proof}
	The initial steps for the proof are similar to those of the proof of Lemma~\ref{lem:proof_sample_init} till \eqref{eq:proof_2}, with the differences being usage of the Lagrangian $\widetilde{	\mathcal{L}}$ and projecting instead onto the convex set $\widetilde{\mathcal{X}}$. For any $\bx \in \widetilde{\mathcal{X}}^n$ and $\blambda \in \mathbb{R}^m$ with $\blambda \succeq \mathbf{0}$, the updates of Algorithm~\ref{algob} satisfy:
	\begin{align} \label{eq:proofb_1}
		\langle \btx^{(t)} - \bx , \bp^{(t)} \rangle & \leq \frac{1}{2\eta} \left( \Vert \bx - \btx^{(t)} \Vert^2 - \Vert \bx - \btx^{(t+1)} \Vert^2 \right) + \frac{\eta}{2} \Vert \bp^{(t)} \Vert^2
	\end{align}
	
	\begin{align} \label{eq:proofb_2}
		\langle \blambda {-} \blambda^{(t)}, \nabla_{\blambda}\widetilde{\mathcal{L}}_t(\bx^{(t)},\blambda^{(t)} \rangle & \leq \frac{1}{2\eta} ( \verts{\blambda {-} \blambda^{(t)}  }^2 {-} \verts{\blambda {-} \blambda^{(t+1)}  }^2 ) + \frac{\eta}{2} \Vert \nabla_{\blambda}\widetilde{\mathcal{L}}_t(\bx^{(t)},\blambda^{(t)}) \Vert^2
	\end{align}
	Now, from the convexity of $\calH(\bx,\blambda^{(t)}) $ in $\bx$, we have:
	\begin{align*}
		\calH(\bx,\blambda^{(t)}) \geq \calH(\bx^{(t)},\blambda^{(t)}) +  \langle \nabla_{\bx} \calH(\bx^{(t)},\blambda^{(t)}), \bx - \bx^{(t)}    \rangle
	\end{align*}
	Similarly, from the concavity of $\calH(\bx^{(t)},\blambda)$ in $\blambda$:
	\begin{align*}
		\calH(\bx^{(t)},\blambda) \leq \calH(\bx^{(t)},\blambda^{(t)}) +  \langle \nabla_{\blambda} \calH(\bx^{(t)},\blambda^{(t)}), \blambda - \blambda^{(t)}    \rangle
	\end{align*}
	The above two equations yield:
	\begin{align*}
		\calH(\bx^{(t)},\blambda) - \calH(\bx,\blambda^{(t)}) 
		\leq \langle \nabla_{\blambda} \calH(\bx^{(t)},\blambda^{(t)}), \blambda - \blambda^{(t)}    \rangle  + \langle \nabla_{\bx} \calH(\bx^{(t)},\blambda^{(t)}), \btx^{(t)} {-} \bx    \rangle {+} \langle \nabla_{\bx} \calH(\bx^{(t)},\blambda^{(t)}), \bx^{(t)} {-} \btx^{(t)}    \rangle
	\end{align*}
	Now we note that $ \nabla_{\blambda} \calH(\bx^{(t)},\blambda^{(t)}) = \nabla_{\blambda} \widetilde{\mathcal{L}}_t (\bx^{(t)},\blambda^{(t)})$ and $\nabla_{\bx} \calH(\bx^{(t)},\blambda^{(t)}) = \bp^{(t)}$. Using bounds from \eqref{eq:proofb_1} and \eqref{eq:proofb_2}:
	\begin{align} \label{eq:proofb_3}
		\calH(\bx^{(t)},\blambda) - \calH(\bx,\blambda^{(t)})
		& \leq \frac{1}{2\eta} \left( \Vert \bx - \btx^{(t)} \Vert^2 - \Vert \bx - \btx^{(t+1)} \Vert^2 \right) + \frac{1}{2 \eta} \left(   \verts{\blambda - \blambda^{(t)}  }^2 - \verts{\blambda - \blambda^{(t+1)}  }^2 \right) + \frac{1}{2\eta} \verts{\bx^{(t)} - \btx^{(t)} }^2  \notag  \\
		& \quad  + \frac{\eta}{2} \left( 2\Vert \bp^{(t)} \Vert^2  + \Vert \nabla_{\blambda}\widetilde{\mathcal{L}}_t(\bx^{(t)},\blambda^{(t)}) \Vert^2  \right) 
	\end{align}

Summing over $t$, and taking expectation yields:
\begin{align} \label{eq:proofb_4}
	\sum_{t=1}^{T} \mathbb{E}	\left[\calH(\bx^{(t)},\blambda) - \calH( \bx,\blambda^{(t)})\right]
	&  \leq 
	\frac{\eta}{2} \sum_{t=1}^{T} \mathbb{E} \left( 2\Vert \bp^{(t)} \Vert^2  + \Vert \nabla_{\blambda}\widetilde{\mathcal{L}}_t(\bx^{(t)},\blambda^{(t)}) \Vert^2  \right) + \frac{1}{2\eta} \sum_{t=1}^{T} \verts{\bx^{(t)} {-} \btx^{(t)} }^2 \notag \\ 
	& \quad + \frac{1}{2 \eta} \sum_{t=1}^{T} \mathbb{E}  \Vert \bx - \btx^{(1)} \Vert^2  
+ \frac{1}{2 \eta} \sum_{t=1}^{T} \mathbb{E} \left(  \verts{\blambda - \blambda^{(t)}  }^2 - \verts{\blambda - \blambda^{(t+1)}  }^2 \right) 
\end{align}
Using the fact $\blambda^{(1)} = \mathbf{0}$, and bounding the term $\verts{\bx - \btx^{(t)}}$ by $4R^2$ using assumption A.\ref{assump:X} completes the proof.

\end{proof}

\begin{fact*} [Restating Fact \ref{prop:bound_f_diff}]
	Under assumptions A.\ref{assump:convexity} and A.\ref{assump:gradient}, for all $t \in [T]$, $i \in [n]$ and any $\bu, \bv \in \mathcal{X}$, we have:
	\begin{align*}
		\bbE _{\bxi_i^{(t)}}[ f_i(\bu , \bxi_i^{(t)}) - f_i(\bv , \bxi_i^{(t)}) ]^2 \leq 4 G_i^2 \verts{\bu - \bv}^2 
	\end{align*}
	where $\bbE_{\bxi_i^{(t)}} [.]$ denotes expectation w.r.t. the sampling at time-step $t$ for the node $i$.
\end{fact*}
\begin{proof}
	From Assumption A.\ref{assump:convexity}, the convexity of $f_i(\bx,\bxi)$ in $\bx$ and $\bxi$ gives:
	\begin{align*}
		& f_i(\bu  , \bxi_i^{(t)}) \geq f(\bv , \bxi_i^{(t)}) +  \langle \nabla_{\bx} f(\bv , \bxi_i^{(t)}), \bu - \bv \rangle \\
		& f_i(\bv, \bxi_i^{(t)}) \geq f(\bu, \bxi_i^{(t)}) + \langle \nabla_{\bx} f(\bu , \bxi_i^{(t)}), \bv - \bu \rangle
	\end{align*}
	The above two equations imply:
	\begin{align*}
		| f_i(\bu , \bxi_i^{(t)}) - f_i(\bv , \bxi_i^{(t)}) |  
		& \leq   \max \{  \langle \nabla_{\bx} f_i(\bu  , \bxi_i^{(t)}), \bu - \bv \rangle , - \langle \nabla_{\bx} f_i(\bv , \bxi_i^{(t)}), \bu - \bv \rangle   \} \\
		& \leq    \left( \verts{  \nabla_{\bx} f_i(\bu , \bxi_i^{(t)})  } + \verts{\nabla_{\bx} f_i(\bv , \bxi_i^{(t)})}   \right) \verts{\bu - \bv}
	\end{align*}
	where the second inequality follows from the Cauchy-Schwarz inequality.
	Taking the square of both sides in the above equation followed by taking the expectation over $\bxi$ and using \eqref{assump:gradient_func} from Assumption A.\ref{assump:gradient}, yields:
	\begin{align*}
		\bbE _{\bxi_i^{(t)}}[ f_i(\bu , \bxi_i^{(t)}) - f_i(\bv , \bxi_i^{(t)}) ]^2 \leq 4 G_i^2 \verts{\bu - \bv}^2  
	\end{align*}
\end{proof}

\begin{prop*} [Restating Proposition \ref{prop:interim_1}] For the update steps given in Algorithm \ref{algob}, under assumptions A.\ref{assump:convexity}-A.\ref{assump:g}, we have:
	\begin{enumerate} [(i)]
		\item $\mathbb{E}\verts{\bp^{(t)}}^2 \leq 4d^2(1+m)G^2 + 4(1+m)\tilde{G}^2 \mathbb{E} \verts{\blambda^{(t)}}^2$
		\item $\mathbb{E} \verts{\nabla_{\blambda} \widetilde{\mathcal{L}}_t(\bx^{(t)},\blambda^{(t)}) }^2 \leq 2C^2 + 2 \delta^2 \eta^2 \mathbb{E} \verts{\blambda^{(t)}}^2$
	\end{enumerate}
\end{prop*}
\begin{proof}
	(i) From the definition of $\bp_i^{(t)}$ in \eqref{eq:bandit_p}:
	\begin{align*}
		\bp_i^{(t)} &:= \frac{d}{2 \zeta} \left[ f_i(\bx_i^{(t)} + \zeta \bu_i^{(t)} , \bxi_i^{(t)} ) - f_i(\bx_i^{(t)} - \zeta \bu_i^{(t)} , \bxi_i^{(t)} )  \right] \bu_i^{(t)}+ 2 \sum_{j \in \mathcal{N}_i} \lambda_{ij}^{(t)} \nabla_{\bx_i} g_{ij} (\bx_i^{(t)},\bx_j^{(t)})
	\end{align*}
	Using the fact $\verts{\bu_i^{(t)}}_2=1$ for all $i \in [n]$ (as they lie on the unit sphere $\mathbb{S}$), we have:
	\begin{align} \label{eq:proof_16}
		 \verts{\bp_i^{(t)}}_2^2   & \leq 2 \verts{\frac{d}{2 \zeta} [ f_i(\bx_i^{(t)} {+} \zeta \bu_i^{(t)} , \bxi_i^{(t)} ) {-} f_i(\bx_i^{(t)} {-} \zeta \bu_i^{(t)} , \bxi_i^{(t)} )  ] }_2^2   + 4 \verts{\sum_{j \in \mathcal{N}_i} \lambda_{ij}^{(t)} \nabla_{\bx_i} g_{ij} (\bx_i^{(t)},\bx_j^{(t)})}_2^2 \notag \\
		& \leq (1{+}|\mathcal{N}_i|) \verts{\frac{d}{2 \zeta} [ f_i(\bx_i^{(t)} {+} \zeta \bu_i^{(t)} , \bxi_i^{(t)} ) {-} f_i(\bx_i^{(t)} {-} \zeta \bu_i^{(t)} , \bxi_i^{(t)} )  ] }_2^2   + 4 (1+|\mathcal{N}_i|) \sum_{j \in \mathcal{N}_i} \left(\lambda_{ij}^{(t)}\right)^2 \verts{ \nabla_{\bx_i} g_{ij} (\bx_i^{(t)},\bx_j^{(t)})}_2^2 
	\end{align}
We now bound the first term in the R.H.S. of \eqref{eq:proof_16}. Using the convexity of $f_i(\bx,\bxi)$ 
in $\bx$ and $\bxi$,
, we note that:
\begin{align*}
	f_i(\bx_i^{(t)} + \zeta \bu_i^{(t)} , \bxi_i^{(t)}) & \geq f(\bx_i^{(t)} - \zeta \bu_i^{(t)} , \bxi_i^{(t)})  + 2 \zeta \langle \nabla_{\bx} f(\bx_i^{(t)} - \zeta \bu_i^{(t)} , \bxi_i^{(t)}), \bu_i^{(t)} \rangle \\
	f_i(\bx_i^{(t)} - \zeta \bu_i^{(t)} , \bxi_i^{(t)}) & \geq f(\bx_i^{(t)} + \zeta \bu_i^{(t)} , \bxi_i^{(t)}) - 2 \zeta \langle \nabla_{\bx} f(\bx_i^{(t)} + \zeta \bu_i^{(t)} , \bxi_i^{(t)}), \bu_i^{(t)} \rangle
\end{align*}
The above two equations imply:
\begin{align*}
	| f_i(\bx_i^{(t)} + \zeta \bu_i^{(t)} , \bxi_i^{(t)}) - f_i(\bx_i^{(t)} - \zeta \bu_i^{(t)} , \bxi_i^{(t)}) |  & \leq 2 \zeta \max \{  -\langle \nabla_{\bx} f_i(\bx_i^{(t)} - \zeta \bu_i^{(t)} , \bxi_i^{(t)}), \bu_i^{(t)} \rangle , \langle \nabla_{\bx} f_i(\bx_i^{(t)} + \zeta \bu_i^{(t)} , \bxi_i^{(t)}), \bu_i^{(t)} \rangle   \} \\
	& \leq   2 \zeta \left( \verts{  \nabla_{\bx} f_i(\bx_i^{(t)} + \zeta \bu_i^{(t)} , \bxi_i^{(t)})  }  + \verts{\nabla_{\bx} f_i(\bx_i^{(t)} - \zeta \bu_i^{(t)} , \bxi_i^{(t)})}   \right) \verts{\bu_i^{(t)}}
\end{align*}
Taking the square of both sides in the above equation followed by taking the expectation over $\bxi$, using \eqref{assump:gradient_func} from Assumption \ref{assump:gradient} and the fact that $\bbE \|\bu_i^{(t)}\|^2 = 1$, yields:
\begin{align} \label{eq:proof_17}
	\bbE [ f_i(\bx_i^{(t)} + \zeta \bu_i^{(t)} , \bxi_i^{(t)}) - f_i(\bx_i^{(t)} - \zeta \bu_i^{(t)} , \bxi_i^{(t)}) ]^2 \leq 16 G_i^2 \zeta^2  
\end{align}
The bound for the second term in the R.H.S. of \eqref{eq:proof_16} can be given by using \eqref{assump:gradient_constraint} in Assumption \ref{assump:gradient} as:
\begin{align} \label{eq:proof_18}
	 4 \sum_{j \in \mathcal{N}_i} \left(\lambda_{ij}^{(t)}\right)^2 \verts{ \nabla_{\bx_i} g_{ij} (\bx_i^{(t)},\bx_j^{(t)})}_2^2 \leq 4 \sum_{j \in \mathcal{N}_i} \left(\lambda_{ij}^{(t)}\right)^2 G_{ij}^2  \leq 4 \widetilde{G}^2 \sum_{j \in \mathcal{N}_i} \left(\lambda_{ij}^{(t)}\right)^2
\end{align}
Plugging in the bounds from \eqref{eq:proof_17} and \eqref{eq:proof_18} in \eqref{eq:proof_16} and taking expectation gives:
\begin{align*}
	\bbE \verts{ \bp_i^{(t)} }^2 \leq (1+|\mathcal{N}_i|) \left[ 4 d^2G_i^2 + 4 \tilde{G}^2 \sum_{j \in \mathcal{N}_i} \bbE \left(\lambda_{ij}^{(t)}\right)^2 \right] 
\end{align*}
Summing over the values of $i$, and using the fact that $|\mathcal{N}_i| \leq m$ and $G^2 =  \sum_{i=1}^n G_i^2 $, we get:
\begin{align*}
	\bbE \verts{ \bp^{(t)} }^2 \leq (1+m) \left[ 4 d^2G^2 + 4 \tilde{G}^2 \bbE \verts{\blambda^{(t)}}^2 \right] 
\end{align*}
This completes the proof for (i) in Proposition \ref{prop:interim_1} \\

(ii)
From the definition of $\widetilde{\mathcal{L}}$ given in \eqref{eq:bandit_vector_lagrangian}, we have:
\begin{align*}
	\widetilde{\mathcal{L}}_t(\bx, \blambda) = \widetilde{f}(\bx,\bxi^{(t)} ) + \langle \blambda, \bg(\bx)  \rangle - \frac{\delta \eta}{2} \verts{ \blambda}_2^2
\end{align*}
Taking gradient w.r.t. $\blambda$ and using the triangle inequality gives:
\begin{align*}
	\verts{  \nabla_{\blambda} \widetilde{\mathcal{L}}_t(\bx^{(t)}, \blambda^{(t)})  } \leq \verts{\bg(\bx^{(t)})} + \delta \eta \verts{ \blambda^{(t)} }
\end{align*} 
Taking square of both sides and expectation yields:
\begin{align*}
	\bbE \verts{  \nabla_{\blambda} \widetilde{\mathcal{L}}_t(\bx^{(t)}, \blambda^{(t)})  }_2^2 \leq 2 \bbE \verts{\bg(\bx^{(t)})}_2^2 + 2\delta^2 \eta^2 \bbE \verts{ \blambda^{(t)} }_2^2
\end{align*}
Using Assumption \ref{assump:g} to bound the first term in the R.H.S. of above equation completes the proof for (ii) in Proposition \ref{prop:interim_1}.

\end{proof}

%

\begin{prop*}[Restating Proposition \ref{prop:interim_2}]
		For any $\blambda \in \mathbb{R}^{m} \text{ with } \blambda \succeq \mathbf{0}$, the updates of Algorithm \ref{algob} satisfy:
	\begin{align*}
		\sum_{t=1}^{T} \mathbb{E}	\left[\calH(\bx^{(t)},\blambda) - \calH((1-\alpha)\bx^{*}+\alpha \tilde{\by}_0,\blambda^{(t)})\right] = \sum_{t=1}^{T} \mathbb{E}	\left[\widetilde{\mathcal{L}}_t(\bx^{(t)},\blambda) - \widetilde{\mathcal{L}}_t((1-\alpha)\bx^{*}+\alpha \tilde{\by}_0,\blambda^{(t)})\right] 
	\end{align*}
	where $\bx^{*}$ is the optimal parameter value for the objective \eqref{eq:main_obj}, and $\mathcal{H}$, $\widetilde{\mathcal{L}}$ are defined in \eqref{eq:bandit_vector_H} and \eqref{eq:bandit_vector_lagrangian} respectively.
\end{prop*}
\begin{proof}
	From the definition of $\calH(\bx,\blambda)$ given in \eqref{eq:bandit_vector_H}, we have:
	\begin{align} \label{eq:interim_1}
		\sum_{t=1}^{T} \mathbb{E}	\left[\calH(\bx^{(t)},\blambda) - \calH((1-\alpha)\bx^{*}+\alpha \tilde{\by}_0,\blambda^{(t)})\right] 
		 & = \sum_{t=1}^{T} \mathbb{E}	\left[\widetilde{\mathcal{L}}_t(\bx^{(t)},\blambda) - \widetilde{\mathcal{L}}_t((1-\alpha)\bx^{*}+\alpha \tilde{\by}_0,\blambda^{(t)})\right] \notag \\
		& + \sum_{t=1}^{T} \bbE \left[  \left\langle \bp^{(t)} {-}  \nabla_{\bx} \widetilde{\mathcal{L}}_t (\bx^{(t)}, \blambda^{(t)}) , \bx^{(t)} {-} (1-\alpha)\bx^{*} {-} \alpha \tilde{\by}_0  \right\rangle \right]
	\end{align}
We now focus on the last term in the R.H.S. of equation \eqref{eq:interim_1}. From the definition of $\bp_i$ given in \eqref{eq:bandit_p}, we have:
\begin{align} \label{eq:interim_2}
	\bp_i^{(t)} & = \frac{d}{2 \zeta} \left[ f_i(\bx_i^{(t)} + \zeta \bu_i^{(t)} , \btheta_i^{(t)}) - f_i(\bx_i^{(t)} - \zeta \bu_i^{(t)} , \btheta_i^{(t)} ) \right] \bu_i^{(t)}  + 2 \sum_{j \in \mathcal{N}_i} \lambda_{ij}^{(t)} \nabla_{\bx_i} g_{ij}(\bx_i^{(t)}, \bx_j^{(t)})
\end{align}
For a given $\bx_i^{(t)}, \btheta_i^{(t)}$, the gradient of the function $\widetilde{f}(\bx_i^{(t)}, \btheta_i^{(t)}) := \bbE_{\bu_i^{(t)} \sim \mathcal{U} (\mathbb{S})} \left[ f_i(\bx_i + \zeta \bu_i^{(t)}, \btheta_i^{(t)}) \right] $ (where $\mathcal{U}(\mathbb{S})$ denotes the uniform distribution on the unit sphere $\mathbb{S}$)  w.r.t. $\bx_i^{(t)}$  for $\zeta >0$ using Fact \ref{fact:bandit_interim1} is:
\begin{align} \label{eq:interim_3}
	\nabla_{\bx_i} \widetilde{f}_i(\bx_i^{(t)}, \btheta_i^{(t)}) & = \frac{d}{2 \zeta} \bbE_{\bu_i^{(t)}} \left[ \left(f_i(\bx_i^{(t)} + \zeta \bu_i^{(t)} , \btheta_i^{(t)} )  - f_i(\bx_i^{(t)} - \zeta \bu_i^{(t)} , \btheta_i^{(t)} )\right) \bu_i^{(t)}   \right]
\end{align}
Taking the expectation w.r.t. $\bu_i^{(t)}$ in \eqref{eq:interim_2} and substituting \eqref{eq:interim_3} in it gives:
\begin{align*}
	\bbE_{\bu_i^{(t)}} \left[\bp_i^{(t)}\right] &= \nabla_{\bx_i}\widetilde{f}_i(\bx_i^{(t)}, \btheta_i^{(t)}) + 2 \sum_{j \in \mathcal{N}_i} \lambda_{ij}^{(t)} \nabla_{\bx_i} g_{ij}(\bx_i^{(t)}, \bx_j^{(t)}) \stackrel{(a)}{=} \nabla_{\bx_i} \widetilde{\mathcal{L}}_t (\bx^{(t)} , \btheta^{(t)})
\end{align*}
where (a) follows from \eqref{eq:bandit_lagrangian_primal_interim1} and the symmetry of dual variables, i.e., $\lambda_{ij}^{(t)} = \lambda_{ij}^{(t)}$ for all $t \in [T]$ and for all $i \in[n], j \in \mathcal{N}_i$. Stacking the vectors $\bbE_{\bu_i^{(t)}} [\bp_i^{(t)}]$ for $i \in [n]$, the expectation of vector $\bp^{(t)}$ over $\{ \bu_i^{(t)} \}_{i=1}^{n}$ is given as:
\begin{align*}
	\bbE_{\bu_1^{(t)} , \hdots , \bu_n^{(t)}} [ \bp^{(t)} ] = \bbE \left[\bp^{(t)} | \bx^{(t)}, \btheta^{(t)}, \blambda^{(t)} \right] = \nabla_{\bx} \widetilde{\mathcal{L}}_t (\bx^{(t)} , \btheta^{(t)})
\end{align*}
Using the equality in above, we note that for any $t \in [T]$
\begin{align*}
	& \bbE \left[  \left\langle \bp^{(t)} -  \nabla_{\bx} \widetilde{\mathcal{L}}_t (\bx^{(t)}, \blambda^{(t)}) , \bx^{(t)} - (1-\alpha)\bx^{*} - \alpha \tilde{\by}_0  \right\rangle \right]\\
	 & = 
	\bbE \left[  \left\langle \bbE [\bp^{(t)} | \bx^{(t)}, \btheta^{(t)}, \blambda^{(t)}] -  \nabla_{\bx} \widetilde{\mathcal{L}}_t (\bx^{(t)}, \blambda^{(t)}) ,  \bx^{(t)} - (1-\alpha)\bx^{*} - \alpha \tilde{\by}_0  \right\rangle \right] \\
	& = 0
\end{align*}
Plugging the above result in \eqref{eq:interim_1} completes the proof.
\end{proof}

\begin{lemma*} [Restating Lemma \ref{lem:e_tb}]  
		Consider the error $\be^{(t)} := \bx^{(t)} - \btx^{(t)}$ for any $t\in[T]$. We have:
	\begin{align*}
		\mathbb{E}\Vert \be^{(t)}\Vert^2 \leq \frac{2 \eta^2}{\omega} \sum_{k=0}^{t-2} \left(  1 - \frac{\omega}{2} \right)^k \mathbb{E}\verts{\bp^{(t-1)}}^2
	\end{align*}
\end{lemma*}
\begin{proof}
	The proof is very similar to that of Lemma~\ref{lem:e_t} in Appendix~\ref{sec:app_feedback}, with the only difference being that $\nabla_{\bx}\mathcal{L}_{t}(\bx^{(t)},\blambda^{(t)} )$ is replaced by $\bp^{(t)}$ and we instead project on the set $\widetilde{\mathcal{X}}^{n}$.
\end{proof}

\end{document}